\pdfoutput=1
\documentclass[11pt]{article}

\makeatletter

\def\inserttitle{}
\def\insertshorttitle{}
\def\insertauthor{}
\def\insertshortauthor{}

\renewcommand{\title}[2][]{%
  \def\inserttitle{#2}%
  \def\@title{#2}%
  \ifx&#1&%
  \def\insertshorttitle{#2}%
  \else%
  \def\insertshorttitle{#1}%
  \fi%
  \sethf
}

\renewcommand{\author}[2][]{%
  \def\insertauthor{#2}%
  \def\@author{#2}%
  \ifx&#1&%
  \def\insertshortauthor{#2}%
  \else%
  \def\insertshortauthor{#1}%
  \fi%
  \sethf
}

\makeatother

\usepackage[twoside,centering,a4paper,bottom=2.7cm,textwidth=15.24cm]{geometry}

\newcommand{\sethf}{%
  \markboth{\hfill\scshape\insertshortauthor\hfill}%
  {\hfill\scshape\insertshorttitle\hfill}
}

\usepackage{amssymb,amsmath,amsthm}

\usepackage{graphicx,color}
\graphicspath{{figures/}}

\usepackage[unicode,colorlinks=true,citecolor=blue,linkcolor=blue]{hyperref}

\renewenvironment{abstract}{\footnotesize\begin{quote}\textbf{Abstract.}~}%
  {\end{quote}}

\theoremstyle{plain}
  \newtheorem{theorem}             {Theorem}    [section]
  \newtheorem{lemma}      [theorem]{Lemma}

  \newtheorem{corollary}  [theorem]{Corollary}
\theoremstyle{definition}
  
\theoremstyle{remark}
  \newtheorem*{remark}             {Remark}
  
\renewcommand{\phi}{\varphi}
\renewcommand{\epsilon}{\varepsilon}
\newcommand{\eps}{\varepsilon}

\DeclareMathOperator{\arcsinh}{arcsinh}
\DeclareMathOperator{\sign}{sign}
\DeclareMathOperator{\trace}{trace}

\renewcommand{\div}{\mathrm{div}}

\newcommand{\bbR}{\mathbb{R}}

\newcommand{\calA}{\mathcal{A}}
\newcommand{\calC}{\mathcal{C}}

\newcommand{\calE}{\mathcal{E}}

\newcommand{\calG}{\mathcal{G}}
\newcommand{\calH}{\mathcal{H}}

\newcommand{\calL}{\mathcal{L}}

\newcommand{\calP}{\mathcal{P}}

\newcommand{\calT}{\mathcal{T}}
\newcommand{\calV}{\mathcal{V}}

\newcommand{\wto}{\rightharpoonup}
\newcommand{\wsto}{\stackrel{*}{\rightharpoonup}}

\newcommand{\loc}{\mathrm{loc}}

\newcommand{\sm}{\setminus}

\newcommand{\set}[2][\empty]{\ensuremath{%
    \left\{%
      \ifx\empty#1%
      \relax%
      \else%
      #1:%
      \fi%
      #2%
    \right\}%
  }%
}
     
\newcommand{\snabla}[2][\empty]{\ensuremath{%
    \ifx\empty#1%
    \nabla #2%
    \else%
    \nabla_{\!\!#1} #2%
    \fi%
  }%
}

\newcommand{\sdiv}[2][\empty]{\ensuremath{%
    \ifx\empty#1%
    \div #2%
    \else%
    \div_{\!#1} #2%
    \fi%
  }%
}

\newcommand{\slaplace}[2][\empty]{\ensuremath{%
    \ifx\empty#1%
    \Delta #2%
    \else%
    \Delta_{#1} #2%
    \fi%
  }%
}

\numberwithin{figure}{section}
\numberwithin{table}{section}
\numberwithin{equation}{section}

\title[An approximation for two-phase lipid bilayers]{%
  Convergence of an approximation for rotationally symmetric two-phase
  lipid bilayer membranes}

\author[M. Helmers]{%
  {\scshape Michael Helmers}\\%
  {\footnotesize Institute for Applied Mathematics, University of Bonn}\\[-.5ex]
  {\footnotesize Endenicher Allee 60, 53115 Bonn, Germany}\\[-.5ex]%
  {\footnotesize Email: {\tt helmers@iam.uni-bonn.de}}%
}

\date{}

\begin{document}

\maketitle
\thispagestyle{empty}

%
%

\begin{abstract}
  We consider a diffuse interface approximation for the lipid phases
  of rotationally symmetric two-phase bilayer membranes and rigorously
  derive its $\Gamma$-limit. In particular, we prove that limit
  vesicles are $C^1$ across interfaces, which justifies a regularity
  assumption that is widely made in formal asymptotic and numerical
  studies. Moreover, a limit membrane may consist of several
  topological spheres, which are connected at the axis of revolution
  and resemble complete buds of the vesicle.


\end{abstract}


\section{Introduction}

Lipid bilayer membranes are an integral part of many biological
systems and display a rich variety of shapes and shape
transformations; in particular, membranes that consist of two or more
lipid phases have a complex morphology affected by the interplay of
elastic properties and phase separation
\cite{SeBeLi91,DoKaNoSpSa93,BaHeWe03}.
The spontaneous curvature model for two-phase lipid bilayer vesicles
describes equilibrium shapes as surfaces minimising the energy
\begin{equation}
  \label{eq:intro:energy}
  \sum_{j=\pm}
  \int_{M^j} k^j ( H-H_s^j )^2 + k_G^j K \,d\mu
  +
  \sigma \calH^1(\partial M^+)
\end{equation}
among all closed surfaces $M = M^+ \cup M^-$, $M^+ \cap M^- =
\emptyset$ with prescribed areas for $M^\pm$ \cite{Canham70, Evans74,
  Helfrich73, JuLi96, SeLi95}. Here $H$ and $K$ are the mean curvature
and the Gauss curvature of the membrane surface $M$, and $\mu$ is its
area measure. The bending rigidities $k^\pm > 0$ and the Gauss
rigidities $k_G^\pm$ are elastic material parameters, and $H_s^\pm$ --
the so-called spontaneous curvatures -- are supposed to reflect an
asymmetry in the membrane. In the simplest case, the rigidities and
spontaneous curvatures are constant within each lipid phase but
different between the two phases. The length of the phase interfaces
$\partial M^+ = \partial M^-$ is denoted by $\calH^1(\partial M^+)$,
and $\sigma$ is a line tension parameter.

In \cite{JuLi96} the Euler-Lagrange equations of
\eqref{eq:intro:energy} for axially symmetric two-phase membranes with
exactly one interface are studied. The authors mention the possibility
of different smoothness conditions at the interface, their analysis,
however, is done for smooth membranes, which are $C^1$ across the
interface, only. Phase field models for the lipid phases and also the
membrane surface are introduced in
\cite{ElSt10_Modeling,ElSt10,DuWa08,LoRaVo09} and studied numerically;
convergence to the sharp interface limit is obtained by asymptotic
expansion and under additional smoothness assumptions and topological
restrictions.

In this paper we are interested in the convergence of a diffuse
interface approximation for the lipid phases in a rotationally
symmetric setting without imposing smoothness or the topological
structure of limit vesicles in advance. More precisely, for a closed
surface $M_\gamma$ obtained by rotating a curve $\gamma$ and an
associated rotationally symmetric phase field $u: M_\gamma \to \bbR$
we consider the approximate energy
\begin{equation}
  \label{eq:intro:approx-energy}
  \int_{M_\gamma}
  k(u) \left( H-H_{s}(u) \right)^2 + k_G(u) K \,d\mu
  +
  \int_{M_\gamma}
  \eps |\snabla[M_\gamma]{u}|^2 + \frac{1}{\eps} W(u) \,d\mu.
\end{equation}
Here the second integral, where $W$ is a standard double well
potential such as $W(u)=(1-u^2)^2$, is the diffuse interface energy
from the Cahn-Hilliard theory of phase transitions \cite{CaHi58} in
the setting of surfaces. As $\eps \to 0$, the phase field is forced to
$\pm 1$, hence the first integral in \eqref{eq:intro:approx-energy}
resembles the curvature integral in \eqref{eq:intro:energy}, provided
that $k(u)$, $k_G(u)$ and $H_s(u)$ are extensions of the given
parameters $k^\pm$, $k_G^\pm$ and $H_s^\pm$.
We prove that, under certain restrictions on these parameters, the
$\Gamma$-limit of \eqref{eq:intro:approx-energy} is given by
\eqref{eq:intro:energy} for rotationally symmetric membranes. In
particular, we obtain that sequences $(\gamma_\eps,u_\eps)$ with
uniformly bounded energy have a subsequence that converges to a limit
membrane consisting of finitely many regular topological spheres,
which are connected at the axis of revolution. By our assumption on
the parameters and the approximation procedure, the limit model has
the property that membranes are $C^1$ across interfaces. For an
approach that allows tangent singularities at interfaces in the limit
see \cite{Helmers13}.

Equi-coercivity and $\Gamma$-convergence also yield the existence of a
minimiser for the limit model. Upon completion of this work, we
became aware of the preprint \cite{ChMoVe}, where the existence of
energy-minimal two-phase membranes in a setting similar to our limit
model is studied and similar issues as in our equi-coercivity and
lower bound arguments are addressed.

The paper is organised as follows. Section \ref{sec:surf} recalls some
facts about surfaces of revolution, in Section \ref{sec:theorem} we
present our setting and state the convergence theorem.  We prove the
theorem in Section \ref{sec:proof} and conclude with some remarks on
generalisations in Section \ref{sec:generalisations}.


\section{Surfaces of revolution}
\label{sec:surf}


\subsection{Basic definitions and notation}

Let $I \subset \bbR$ be an open bounded interval and $\gamma=(x,y)
\colon I \to \bbR^2$ a Lipschitz parametrised curve in the upper half
of the $x y$-plane, that is, $y(t) \geq 0$ for all $t \in I$. We
denote by $M_\gamma$ the surface in $\bbR^3$ obtained by rotating
$\gamma$ about the $x$-axis, thus $M_\gamma$ is the image of
$\overline{I} \times [0,2\pi)$ under the Lipschitz continuous map
\begin{equation*}
  \Phi \colon
  (t,\theta) \mapsto (x(t), y(t) \cos\theta, y(t) \sin\theta);
\end{equation*}
$\gamma$ is called generating curve of
$M_\gamma$. See~\cite{doCarmo76,Kuehnel99translated} for a detailed
discussion of surfaces.

Since $\gamma$ is Lipschitz continuous, it is weakly and almost
everywhere differentiable with bounded derivative $\gamma'$, and the
Fundamental Theorem of Calculus
\begin{equation*}
  \gamma(t_1) - \gamma(t_0)
  =
  \int_{t_0}^{t_1} \gamma'(t) \,d t
\end{equation*}
holds for all $t_0, t_1 \in \overline{I}$. The length of $\gamma$
is given by
\begin{equation*}
  \calL_{\gamma} =  \int_{I} |\gamma'(t)| \,d t,
\end{equation*}
and after removing at most countably many constancy intervals, pulling
holes together and reparametrising, we may assume that $\gamma$ is
parametrised with constant speed $|\gamma'| \equiv \calL_\gamma/|I| =:
q_\gamma$ almost everywhere in $I$~\cite[Lemma 5.23]{BuGiHi98}.

By $\mu$ we denote the area measure of $M_\gamma$, that is $d\mu =
|\partial_t \Phi \wedge \partial_\theta \Phi| \,d t \,d\theta =
|\gamma'| y \,d t \,d\theta$, and we write
\begin{equation*}
  \calA_\gamma
  =
  \int_{M_\gamma} \,d\mu
  =
  2\pi \int_{I} |\gamma'| y \,d t
\end{equation*}
for the area of $M_\gamma$. Moreover, for a measurable subset $J$ of
$I$, we let $M_\gamma(J)$ be the part of $M_\gamma$ that is obtained
by rotating the curve segment $\gamma(J)$, and refer to the
corresponding length and area as $\calL_\gamma(J)$ and
$\calA_\gamma(J)$, respectively. If $\gamma$ is embedded, then also
$M_\gamma$ is, and $\mu$ is the two-dimensional Hausdorff measure
$\calH^2$ restricted to $M_\gamma$; in general, however, the
multiplicity of $\mu$ may be larger than $1$.

The tangent space $\calT_{(t_0,\theta_0)} M_\gamma$ exists for almost
every $(t,\theta) \in I \times [0,2\pi)$ and is the plane spanned
by the orthonormal vectors
\begin{equation}
  \label{eq:surf:orthonormal-basis}
  \xi_1
  =
  \frac{\partial_t \Phi}{|\partial_t \Phi|}
  =
  \frac{1}{|\gamma'|} \left( x', y' \cos\theta, y' \sin\theta \right)
  \qquad\text{and}\qquad
  \xi_2
  =
  \frac{\partial_\theta \Phi}{|\partial_\theta \Phi|}
  =
  \left( 0, -\sin\theta, \cos\theta \right);
\end{equation}
a unit normal is given by
\begin{equation}
  \label{eq:surf:unit-normal}
  \nu
  =
  \frac{\partial_t \Phi \wedge \partial_\theta \Phi}
  {|\partial_t \Phi \wedge \partial_\theta \Phi|}
  =
  \frac{1}{|\gamma'|} \left( -y', x' \cos\theta, x' \sin\theta \right).
\end{equation}
We associate tangent space, normal and all other geometric quantities
to the parameter $(t,\theta)$ and not to the point $\Phi(t,\theta)$ on
the surface $M_\gamma$, because $M_\gamma$ is not necessarily
embedded. For the same reason, we consider a function $f \colon
M_\gamma \to \bbR^k$ to be a function $F \colon \overline{I} \times
[0,2\pi) \to \bbR^k$ of the parameters; on embedded parts of
$M_\gamma$ this amounts to $f(\Phi(t,\theta)) = F(t,\theta)$. Given a
tangent vector $\xi$ at $(t_0,\theta_0) \in I \times (0,2\pi)$, the
directional derivative of $f$ in direction $\xi$ is defined as
\begin{equation*}
  D_\xi f(t_0,\theta_0)
  = \left. \frac{d}{d s} F(\eta(s)) \right|_{s=0},
\end{equation*}
where $\eta \colon (-\delta,\delta) \to I \times [0,2\pi)$ is a
$C^1$-curve satisfying $\eta(0) = (t_0,\theta_0)$ and $ \frac{d}{d s}
\Phi(\eta(s)) \big|_{s=0} = \xi$. The tangential gradient of $f
\colon M_\gamma \to \bbR$ is
\begin{equation*}
  \snabla[M_\gamma]{f} = (D_{\xi_1} f) \xi_1 + (D_{\xi_2} f) \xi_2,
\end{equation*}
where $\set{\xi_1,\xi_2}$ is any orthonormal basis of the tangent
space; see \cite{Simon83, Kuehnel99translated} for a detailed
discussion. For $\set{\xi_1,\xi_2}$ as in
\eqref{eq:surf:orthonormal-basis}, we find $D_{\xi_1} f =
|\gamma'|^{-1} \partial_t F$ and $D_{\xi_2} f = y^{-1} \partial_\theta
F$, hence
\begin{equation*}
  \snabla[M_\gamma]{f}
  =
  \frac{1}{|\gamma'|} (\partial_t F) \xi_1
  +
  \frac{1}{y} (\partial_\theta F) \xi_2
  =
  \frac{1}{|\gamma'|^2} (\partial_t F) \partial_t \Phi
  +
  \frac{1}{y^2} (\partial_\theta F) \partial_\theta \Phi.
\end{equation*}
In particular, if $f$ is rotationally symmetric, which means that it
is independent of $\theta$, then
\begin{equation*}
  \snabla[M_\gamma]{f}(t,\theta)
  = \frac{F'(t)}{|\gamma'(t)|} \xi_1(t,\theta)
  \qquad \text{and} \qquad
  |\snabla[M_\gamma]{f}(t,\theta)| = \frac{|F'(t)|}{|\gamma'(t)|},
\end{equation*}
where $| \cdot |$ is the Euclidean norm in $\bbR^3$.


For the rest of this subsection let $\gamma \in W^{2,1}_{\loc}(I;
\bbR^2)$ be twice weakly differentiable, thus twice differentiable
almost everywhere, and $y>0$ in $I$. Since $\nu$ in
\eqref{eq:surf:unit-normal} is weakly differentiable, the shape
operator $S \colon \calT_{(t_0,\theta_0)}M \to
\calT_{(t_0,\theta_0)}M$, $\zeta \mapsto D_\zeta \nu$ and the second
fundamental form $B \colon \calT_{(t_0,\theta_0)}M \times
\calT_{(t_0,\theta_0)}M \to \bbR$, $(\zeta,\xi) \mapsto \xi \cdot
D_\zeta \nu$ are well-defined for almost every $(t_0,\theta_0)$. The
matrix representation with respect to the basis $\set{\xi_1,\xi_2}$ in
\eqref{eq:surf:orthonormal-basis} of both is
\begin{equation*}
  \begin{pmatrix}
    \kappa_1 & 0 \\
    0        & \kappa_2
  \end{pmatrix}
  \qquad\text{with}\qquad
  \kappa_1
  =
  \frac{-y'' x' + y' x''}{|\gamma'|^3}
  =
  - \frac{\gamma'' \cdot \gamma'^{\perp}}{|\gamma'|^3}
  \quad\text{and}\quad
  \kappa_2
  =
  \frac{x'}{y|\gamma'|}.
\end{equation*}
The eigenvalues $\kappa_1, \kappa_2$ of $S$ are the principal
curvatures of $M_\gamma$, and $\kappa_1$ is just the signed curvature
of $\gamma$ with respect to the normal $-\gamma'^{\perp}/|\gamma'| =
(y',-x')/|\gamma'|$. The mean curvature $H$ and the Gauss curvature
$K$ of $M_\gamma$ are
\begin{equation*}
  H = \trace S = \kappa_1 + \kappa_2
  \qquad\text{and}\qquad
  K = \det S = \kappa_1 \kappa_2.
\end{equation*}
By $|S|^2 = \kappa_1^2 + \kappa_2^2$ we denote the squared Frobenius
norm of $S$, and since $B(\zeta,\zeta) = \zeta \cdot S \zeta$, we also
write $|B|^2 = |S|^2$. Obviously, we have $|B|^2 = H^2 - 2K$.

The signs of the principal curvatures and the mean curvature depend on
the sign of the normal $\nu$. Our choice above ensures that a unit
ball has outer unit normal $\nu$ as in \eqref{eq:surf:unit-normal} and
curvatures $\kappa_1=\kappa_2=+1$ when its generating curve is
parametrised ``from left to right'' such that $x' \geq 0$, for
instance by $\gamma(t) = ( -\cos t, \sin t)$, $t \in [0,\pi]$.

Let $\phi \colon I \to \bbR$ be an angle function for $\gamma$, that
is, let $\phi(t)$ be the angle between the positive $x$-axis and the
tangent vector $\gamma'(t)$. Since $W^{2,1}_{\loc}$ embeds into
$C^1_{\loc}$, the angle $\phi$ can be chosen continuously in $I$ and
is then uniquely determined up to multiples of $2\pi$. In terms of
$\phi$, the curve $\gamma$ is characterised by fixing one point and
\begin{equation*}
  x' = |\gamma'| \cos \phi,
  \qquad
  y' = |\gamma'| \sin \phi.
\end{equation*}
The principal curvatures take the form
\begin{equation*}
  \kappa_1 = - \frac{\phi'}{|\gamma'|},
  \qquad
  \kappa_2 = \frac{\cos \phi}{y},
\end{equation*}
and we have
\begin{equation}
  \label{eq:surf:gauss-curv-angle}
  K
  =
  - \frac{\phi' \cos \phi}{|\gamma'| y}
  =
  - \frac{\left(\sin \phi \right)'}{|\gamma'| y}
  =
  - \frac{\left( y' / |\gamma'| \right)'}{|\gamma'|y}.
\end{equation}
From \eqref{eq:surf:gauss-curv-angle} we see that for any $J = (a,b)
\Subset I$ the integral
\begin{equation}
  \label{eq:surf:gauss-bonnet}
  \int_{M_\gamma(J)} K d\mu
  =
  - 2\pi \int_a^b \left(\sin \phi \right)' \,d t
  =
  2\pi \left( \sin \phi(a) - \sin \phi(b) \right).
\end{equation}
depends only on the tangent angle at $\partial J$. If additionally
$\phi \in C^0(\overline{I})$, then \eqref{eq:surf:gauss-bonnet} is by
approximation also true for $J = I$, which is just the Gauss Bonnet
Theorem for surfaces of revolution. In particular, if $y(\partial I) =
\set{0}$ and $M_\gamma$ is a $C^1$-surface, then $\gamma'$ is
perpendicular to the axis of revolution at $\partial I$ and we
conclude $\int_{M_\gamma} K \,d\mu = 4\pi$.

Another consequence of~\eqref{eq:surf:gauss-curv-angle} is that for
$\gamma$ parametrised with constant speed $q_\gamma>0$ the integral
\begin{equation}
  \label{eq:surf:gauss-and-gamma}
  \int_{M_\gamma(J)} |K| \,d\mu
  =
  \frac{2\pi}{q_\gamma} \int_J |y''| \,d t
\end{equation}
is the $L^1$-norm of $y''$. Moreover, in that case we also have
$|\gamma''|^2 = \phi'^2 q_\gamma^2$ and obtain that
\begin{equation}
  \label{eq:surf:kappa1-and-gamma}
  \int_{M_\gamma(J)} \kappa_1^2 \,d\mu
  =
  \frac{2\pi}{q_\gamma} \int_J |\phi'|^2 y \,d t
  =
  \frac{2\pi}{q_\gamma^3} \int_J |\gamma''|^2 y \,d t
\end{equation}
is a weighted $L^2$-norm of $\phi'$ and $\gamma''$.

If $M_\gamma$ is a closed surface, that is $y(\partial I) = \set{0}$,
$\kappa_2$ seemingly degenerates at the axis of revolution. However,
if $M_\gamma$ is sufficiently smooth, the principal curvatures are
still well-defined, for instance by taking another local
parametrisation of $M_\gamma$; to compute $\kappa_2$ in the
rotationally symmetric parametrisation, L'H{\^o}pital's rule may be
used and yields $\kappa_2 = \kappa_1$ \cite{Kuehnel99translated}.


\subsection{Surfaces with
  \texorpdfstring{$L^2$}{L\unichar{"00B2}}-bounded second
  fundamental form}

The sharp inequality $y>0$ in $I$ is not conserved by the convergence
of curves that our $\eps$-energy yields. If merely $y \geq 0$ in $I$,
the set $\set{y>0} = \set[t \in I]{y(t)>0}$ is open and hence is the
union of its countably many connected components, which are disjoint
open intervals. In a slight abuse of language we also call
$M_\gamma(\omega)$ a component of $M_\gamma$ if $\omega$ is a
component of $\set{y>0}$. Thus, $M_\gamma$ consists of at most
countably many components, which are connected at the axis of
revolution.

In the following lemma and corollary we collect some regularity
properties of $\gamma$ and $M_\gamma$ that follow from an $L^2$-bound
on the second fundamental form. The focus here is on regions of
$\set{y>0}$ near the axis of revolution.

\begin{lemma}
  \label{lem:surf:sec-fform-bound}
  Let $\gamma = (x,y) \colon I \to \bbR^2$, $y \geq 0$ be a Lipschitz
  curve that satisfies $\gamma \in W^{2,1}_{\loc}(\set{y>0}; \bbR^2)$,
  $|\gamma'| \equiv q_\gamma >0$ in $\set{y>0}$, and
  \begin{equation*}
    \int_{M_\gamma(\set{y>0})} |B|^2 \,d\mu < \infty.
  \end{equation*}
  Then we have $\gamma \in W^{2,2}_{\loc}(\set{y>0}; \bbR^2)$ and $y
  \in W^{2,1}(\set{y>0}; \bbR^2)$.
  Moreover, for any connected component $\omega = (a,b)$ of
  $\set{y>0}$ the curve $\gamma$ belongs to $C^1(\overline\omega;
  \bbR^2)$ and has one-sided derivatives $\gamma'(a) = - \gamma'(b) =
  (0,|\gamma'|)$, which means that $\gamma$ is perpendicular to the
  axis of revolution.
  The number of components of $\set{y>0}$ is finite.
\end{lemma}

\begin{proof}
  On any set $J \Subset \set{y>0}$ the $y$-coordinate has a positive
  lower bound $c_J$ in $J$, thus $\gamma \in W_\loc^{2,2}(\{y>0\};
  \bbR^2)$ follows from \eqref{eq:surf:kappa1-and-gamma}; using $2|K|
  \leq |B|^2$ and \eqref{eq:surf:gauss-and-gamma} we obtain $y \in
  W^{2,1}(\set{y>0})$. The Sobolev embedding theorem then yields $x\in
  C^1_{\loc}(\omega)$ and $y \in C^1(\overline \omega)$ for any
  connected component $\omega=(a,b)$ of $\set{y>0}$, and we aim to
  show that also $x \in C^1(\overline \omega)$.

  Assume for contradiction that there are sequences $t_k \to a$, $s_k
  \to a$ in $\omega$ such that $\lim x'(t_k) \not= \lim x'(s_k)$; if
  such sequences cannot be found, $x'(t)$ converges as $t\searrow a$.
  Since $x'(t_k)^2$ and $x'(s_k)^2$ converge to $q_\gamma^2-y'(a)^2$,
  we have $\lim x'(s_k) = - \lim x'(t_k) = m \not= 0$ and $x'(t_k) <
  -m/2$ and $x'(s_k) > m/2$ for sufficiently large $k$. Thus, there is
  $r_k \in (t_k,s_k)$ or $(s_k,t_k)$ such that $x'(r_k)=0$, and from
  $r_k \to a$ we infer that $y'^2(a) = q_\gamma^2$. Consequently, we
  find $x'^2(s_k) = q_\gamma^2-y'^2(s_k) \to 0$ and $x'^2(t_k) =
  q_\gamma^2-y'^2(t_k) \to 0$ in contradiction to our
  assumption. Since the same argument applies at $t=b$, we obtain $x'
  \in C^1(\overline \omega)$.

  Next, to prove that $\gamma$ is perpendicular to the axis of
  revolution at $a$, we use $y(t) \leq q_\gamma (t-a)$ in $\omega$ and
  the second principle curvature of $M_\gamma$ to deduce that
  \begin{equation*}
    \infty
    >
    \frac{q_\gamma ^2}{2\pi} \int_{M(\omega)} \kappa_2^2 \,d\mu
    \geq
    \int_a^{a+\delta} \frac{x'^2}{t-a} \,d t
    \geq
    \left( \inf_{(a,a+\delta)} x'^2 \right) \int_a^{a+\delta} \frac{d t}{t-a}
  \end{equation*}
  for all $\delta \in (0,b-a)$. Continuity of $x'$ now implies
  $x'(a) = 0$, and similarly we get $x'(b)=0$. As $|\gamma'| =
  q_\gamma$ and $y>0$ in $\omega$, we find $y'(a)=-y'(b) =
  q_\gamma$.

  Finally, by the Gauss-Bonnet formula \eqref{eq:surf:gauss-bonnet} we
  have
  \begin{equation*}
    \int_{M(\omega)} K \,d\mu
    =
    4\pi
  \end{equation*}
  for each component $\omega$ of $\set{y>0}$, so the number $N_\gamma$
  of components of $\set{y>0}$ satisfies
  \begin{equation*}
    N_\gamma
    \leq
    \frac{1}{4\pi} \sum_{\omega} \int_{M_\gamma(\omega)} |K| \,d\mu
    \leq
    \frac{1}{8\pi} \int_{M_\gamma(\set{y>0})} |B|^2 \,d\mu
  \end{equation*}
  and is thus finite.
\end{proof}

\begin{corollary}
  Let $\gamma=(x,y)$ be as in Lemma \ref{lem:surf:sec-fform-bound}.
  Then $M_\gamma$ has finitely many components which are connected at
  the axis of revolution. Each component is an immersed $C^1$-surface
  and a $W^{2,2}$-surface in $\set{y>0}$, that is, away from the axis
  of revolution.
\end{corollary}

\begin{remark}
  The properties $y \in W^{2,1}(\set{y>0})$ and $x \in
  C^1(\set{y>0})$, but $x'\not\in W^{2,1}(\set{y>0})$ in
  Lemma~\ref{lem:surf:sec-fform-bound} are sharp, as the following
  example shows. Let
  \begin{equation*}
    \psi(t)
    =
    \frac{ \sin \ln (1/t) +1 }{ \ln (1/t) }
  \end{equation*}
  for $t \in (0,t_0)$ with $t_0$ sufficiently small that $\psi(t) \in
  [0,1]$ for all $t \in (0,t_0)$ and consider
  \begin{equation*}
    x'(t)
    =
    \cos \left( \pi/2 - \psi(t) \right)
    =
    \sin \psi(t),
    \qquad
    y'(t)
    =
    \sin \left( \pi/2 - \psi(t) \right)
    =
    \cos \psi(t)
  \end{equation*}
  with $x(0)=y(0)=0$. As $t \to 0$, $\psi(t)$ converges to $0$ and we
  have $x'(t) \sim \psi(t)$, $y'(t) \sim 1$, and $y(t) \sim t$ for all
  small $t$, where $a \lesssim b$ denotes $a \leq Cb$ with a constant
  $C>0$ and $a \sim b$ means $a \lesssim b \lesssim a$. Thus we obtain 
  \begin{equation*}
    \int \kappa_2^2 \,d\mu
    \sim
    \int \frac{\psi^2}{t} \,d t
    \lesssim
    \int \frac{ d t }{ t (\ln(1/t))^2 }
    <
    \infty.
  \end{equation*}
  Moreover, the derivative of $\psi$ is
  \begin{equation*}
    \psi'(t)
    =
    - \frac{ \cos \ln(1/t) }{ t \ln(1/t) }
    + \frac{ \sin \ln(1/t) + 1 }{ t (\ln(1/t))^2 },
  \end{equation*}
  which implies
  \begin{equation*}
    \int \kappa_1^2 \,d\mu
    \sim
    \int \psi'^2 t \,d t
    \lesssim
    \int \frac{ d t }{ t (\ln(1/t))^2 }
    <
    \infty.
  \end{equation*}
  On the other hand, we have $x'' = \psi' \cos \psi \sim \psi'$ for
  small $t$ and
  \begin{equation*}
    \int |\psi'| \,d t
    \sim
    \int \frac{ |\cos \ln(1/t)| }{ t \ln(1/t)} \,d t
    =
    \infty,
  \end{equation*}
  thus $x \not\in W^{2,1}((0,t_0))$. Furthermore, $y'' = -\psi' \sin \psi
  \sim -\psi' \psi$ and
  \begin{equation*}
    \int |\psi' \psi|^p d t
    \sim
    \int \frac{ d t }{ t^p (\ln(1/t))^{2p} }
    =
    \infty
  \end{equation*}
  for any $p>1$ yield $y'' \not\in L^p((0,t_0))$. Note that $M_\gamma$
  is embedded due to $x' \geq 0$.
\end{remark}


\subsection{Length bound}

To establish compactness of energy bounded sequences, we need bounds
on the curves that are derived from bounds on the curvature integrals
in the energy. Two such results, which are well-known and valid for
arbitrary smoothly immersed surfaces, are \cite[Lemma 1.1]{Simon83}
and \cite{Topping08}, which relate the extrinsic and intrinsic
diameter of a surface to its mean curvature. The proof of both
results hinges on the fact that in an arbitrary ball the mean
curvature and the area cannot be small at the same time; the diameter
bounds are then obtained by a covering argument.
For closed surfaces of revolution, however, there is a straightforward
proof that the mean curvature integral bounds the length of the
generating curve.

\begin{lemma}
  \label{lem:surf:length-bound}
  Let $\gamma = (x,y) \in C^{0,1}(I;\bbR^2) \cap
  W^{2,1}_{\loc}(I;\bbR^2)$ be a curve such that $y(I) \subset
  (0,\infty)$ and $y(\partial I) = \set{0}$. Then
  \begin{equation*}
    \int_{M_\gamma} |H| d\mu
    \geq
    2\pi \calL_\gamma.
  \end{equation*}
\end{lemma}

\begin{proof}
  We may assume that the mean curvature integral is finite because
  otherwise there is nothing to prove. Without loss of generality we
  also assume that $\gamma \colon (0,\calL_\gamma) \to \bbR^2$ is
  parametrised by arc length. If $x' \geq 0$ in $I$, there is an
  angle $\phi$ that is weakly differentiable in $I$ and satisfies
  $\phi \in [-\pi/2,\pi/2]$. Then we obtain
  \begin{align*}
    \int_{M_\gamma} H \,d\mu
    &=
    2\pi \int_{0}^{\calL_\gamma}
    \left(-\phi' + \frac{\cos \phi}{y} \right) y \,d t
    =
    2\pi \int_{0}^{\calL_\gamma}
    \phi y' + \cos \phi \,d t - \left. 2\pi \phi y \right|_{0}^{\calL_\gamma}
    \\
    &=
    2\pi \int_{0}^{\calL_\gamma} \phi \sin \phi + \cos \phi \,d t
    \\
    &\geq
    2\pi \calL_\gamma,
  \end{align*}
  because $\phi \sin \phi + \cos \phi \geq 1$.
  In general, when $x' \geq 0$ does not hold, we consider the curve
  $\widetilde\gamma = (\widetilde x,\widetilde y)$ defined by
  \begin{equation*}
    \widetilde y = y
    \qquad\text{and}\qquad
    \widetilde x(t) = x(0) + \int_0^t |x'(s)| \,d s.
  \end{equation*}
  We clearly have $|\widetilde \gamma'| = |\gamma'|=1$, and a simple
  calculation shows $\widetilde H = H \sign x'$ almost
  everywhere. Therefore, we conclude
  \begin{equation*}
    \int_{M_\gamma} |H| \,d\mu
    \geq
    \int_{M_{\widetilde\gamma}} \widetilde H \,d\widetilde\mu
    \geq
    2\pi \calL_{\widetilde \gamma}
    =
    2\pi \calL_\gamma.
    \qedhere
  \end{equation*}
\end{proof}

\begin{remark}
  The inequality in Lemma~\ref{lem:surf:length-bound} is actually
  strict: equality in the above calculation means $\phi \sin \phi +
  \cos \phi = 1$, which holds only if $\phi \equiv 0$ and is thus
  impossible for a nontrivial closed surface of revolution. Moreover,
  the inequality is sharp, as can be seen by a cylinder with spherical
  caps when the radius tends to zero.
\end{remark}


\section{Energies and \texorpdfstring{$\Gamma$}{\textGamma}-convergence}
\label{sec:theorem}


\subsection{Approximate setting}

Recall from the introduction that we aim to approximate the energy
\eqref{eq:intro:energy} by
\begin{equation}
  \label{eq:thm:approx-energy}
  \calE_\eps(\gamma,u)
  =
  \int_{M_\gamma}
  k(u) \left( H-H_{s}(u) \right)^2 + k_G(u) K \,d\mu
  +
  \int_{M_\gamma}
  \eps |\snabla[M_\gamma]{u}|^2 + \frac{1}{\eps} W(u) \,d\mu.
\end{equation}
The prescribed areas of the lipid phases translate into constraints
on the area of $M_\gamma$ and on the phase integral $\int_{M_\gamma} u
\,d\mu$: if the areas of the lipid phases are given by $A^+$ and
$A^-$, we require that
\begin{equation*}
  \calA_\gamma = A_0 := A^+ + A^-
  \qquad\text{and}\qquad
  \int_{M_\gamma} u \,d\mu = m A_0, \;\text{where }
  m = (A^+ - A^-)/A_0.
\end{equation*}
We assume that the double well potential $W \colon \bbR \to
[0,\infty)$ is a continuous function that vanishes only in $\pm1$ and,
for technical reasons, is $C^2$ around these points. We let $H_s
\colon \bbR \to\bbR$ be a continuous and bounded extension of the
spontaneous curvatures $H_s^\pm \in \bbR$ such that
$H_s(\pm1)=H_s^\pm$. For the bending rigidities $k^\pm > 0$ we suppose
that $k \colon \bbR \to \bbR$ is a continuous and bounded extension of
$k(\pm1) = k^\pm$ satisfying
\begin{equation*}
  \inf_{u \in \bbR} k(u) =: k_0 > 0
\end{equation*}
and that $\widetilde k_G \colon \bbR \to (-\infty,0]$ is a bounded
and continuous extension of $\widetilde k_G(\pm1) = k_G^\pm \leq 0$.
Moreover, we require
\begin{equation}
  \label{eq:thm:kg-assumption}
  k(u)
  >
  - \frac{\widetilde k_G(u)}{2}
  \geq
  - \frac{k_G(u)}{2}
  \qquad\text{uniformly in } u \in \bbR,
\end{equation}
where
\begin{equation*}
  k_G(u) = \min(u^2,1) \widetilde k_G(u).
\end{equation*}

Experimental measurements of the Gauss rigidity are scarce, but
available data suggest that for some membranes $-1 < k_G/(2k) < 0$ and
thus \eqref{eq:thm:kg-assumption} are satisfied
\cite{SiKo04,TeKhSe98}.  Furthermore, our assumptions are
mathematically motivated by the inequalities
\begin{align}
  \label{eq:thm:cond1}
  \calE_\eps(\gamma,u)
  &\geq
  C \int_{M_\gamma} |H|^2 \,d\mu - C,
  \\
  \label{eq:thm:cond2}
  k(u) \left(H-H_s(u)\right)^2 + k_G(u) K
  &\geq
  - C,
\end{align}
which are necessary to obtain a suitable compactness result and the
$\Gamma$-convergence lower bound; here $C$ is a generic constant
independent of $\gamma$. Indeed, expanding the quadratic term on the
left hand side of \eqref{eq:thm:cond2} and applying Young's inequality
with some $\delta>0$ to the mixed term $2 H H_{s}$ yields
\begin{align*}
  k(u) &\left( H-H_{s}(u) \right)^2 + k_G(u) K
  \\
  &=
  - \frac{k_G(u)}{2} |B|^2 
  + \left( k(u) + \frac{k_G(u)}{2} \right) H^2
  + k(u) H_{s}(u)^2
  - 2 k(u) H H_{s}(u)
  \\
  &\geq
  - \frac{k_G(u)}{2} |B|^2 
  + \left( k(u) (1-\delta) + \frac{k_G(u)}{2} \right) H^2
  - k(u) \frac{1-\delta}{\delta} H_{s}(u)^2,
\end{align*}
hence, \eqref{eq:thm:cond2} is satisfied, provided that $k_G(u) \leq
0$ and \eqref{eq:thm:kg-assumption} hold. Then \eqref{eq:thm:cond1} is
true, if additionally the area of $M_\gamma$ is prescribed.

Most interesting is the factor $u^2$ in our definition of $k_G$, which
differs from other diffuse models for the lipid phases
\cite{ElSt10,DuWa08} where the extended Gauss rigidities are bounded
away from $0$. The latter studies do not consider topological changes
in the limit, which, however, is necessary to establish an
equi-coercivity result. The purpose of the $u^2$ is to allow the
construction of appropriate recovery sequences; see the end of Section
\ref{sec:curve-appr-recov} for the details.

We study \eqref{eq:thm:approx-energy} for membranes $(\gamma,u) \in
\calC_\eps \times \calP_\eps$, where
\begin{equation*}
  \begin{split}
  \calC_{\eps}
  :=
  \Big\{
    &\gamma = (x,y) \in C^{0,1}(I; \bbR^2)
    \cap W^{2,1}_{\loc}(I; \bbR^2) :\\
    &|\gamma'| = \text{const}, \;
    y(\partial I)=\set{0}, \;
    y(I) \subset (0,\infty),
    \int_{M_\gamma} |B|^2 \,d\mu < \infty, \;
    \calA_\gamma = A_0
  \Big\}
  \end{split}
\end{equation*}
and
\begin{equation*}
  \calP_\eps
  :=
  \Big\{
    u \in W^{1,1}_{\loc}(I) :
    \int_{M_\gamma} |\snabla[M_\gamma]{u}|^2 \,d\mu < \infty, \;
    \| u \|_\infty \leq C_0, \;
    \int_{M_\gamma} u \,d\mu = m A_0
  \Big\}.
\end{equation*}
The first three conditions in the definition of $\calC_\eps$ ensure
that $\gamma$ is parametrised with constant speed and that $M_\gamma$
is a closed surface. The $L^2$-bound on the second fundamental form of
$M_\gamma$ together with the first two conditions on the phase fields
ensure that the energy \eqref{eq:thm:approx-energy} is well-defined on
$\calC_\eps \times \calP_\eps$. The requirement $\|u\|_\infty \leq
C_0$ with a large constant $C_0 \gg 1$ seems rather strong, but phase
fields with small energy are expected to be close to the interval
$[-1,1]$ anyway. In fact, in many places in the proof the
$L^\infty$-bound can be replaced by a less restrictive condition;
compare \cite[Proposition 3]{Modica87}.

Although the set $\calP_\eps$ depends on the chosen $\gamma \in
\calC_\eps$ via the phase area constraint, we suppress this fact in
the notation, because we usually consider pairs or membranes
$(\gamma,u)$. Instead, we highlight the affiliation to the approximate
energy by the index $\eps$ in $\calC_\eps \times \calP_\eps$. In the
following, we write $M_\eps$ instead of $M_{\gamma_\eps}$ and so
forth, when considering sequences $(\gamma_\eps)$ of curves. If
necessary or useful for clarification we add the curve or $\eps$ as
index to other quantities such as $H_\gamma$, $\mu_\gamma$, or
$y_\gamma$.

The energy \eqref{eq:thm:approx-energy} is invariant under
reparametrisations that preserve the orientation and the regularity
properties of $\gamma$. In particular, if $(\gamma,u)$ satisfies all
requirements of $\calC_\eps \times \calP_\eps$ but only $|\gamma'|
\not=0$ instead of $|\gamma'| = \text{const}$, the corresponding
constant speed parametrisation belongs to $\calC_\eps \times
\calP_\eps$ and has the same energy. Hence, considering only
$|\gamma'| = \text{const}$ is no geometric restriction.


\subsection{Limit setting}

Our limit energy is
\begin{equation*}
  \calE(\gamma,u)
  =
  \int_{M_\gamma} k(u) (H-H_{s}(u))^2 + k_G(u) K \,d\mu
  +
  \sigma \calH^1(M_\gamma(S_u))
\end{equation*}
for curves with parametrisations $\gamma$ in
\begin{equation*}
  \begin{split}
  \calC
  := \Big\{
    &\gamma = (x,y) \in
    C^{0,1}(I;\bbR^2)
    \cap W^{2,1}_{\loc}(\set{y>0}; \bbR^2) :
    \\
    &|\gamma'| = \text{const}, \;
    y(\partial I) = \set{0}, \;
    y \geq 0, \;
    \calH^0(\set{y=0}) < \infty, \;
    \\
    &\int_{M(\set{y>0})} |B|^2 \,d\mu < \infty, \;
    \calA_\gamma = A_0
  \Big\}
\end{split}
\end{equation*}
and associated phase fields $u$ in
\begin{equation*}
  \calP
  :=
  \Big\{
    u \colon I \to \set{\pm1} \text{piecewise constant}:
    \int_M u \,d\mu = m A_0,\;    \calH^1(M_\gamma(S_u)) < \infty
  \Big\}.
\end{equation*}
Here $S_u \subset \set{y>0}$ denotes the countable jump set of $u$ in
$\set{y>0}$, and we call $s \in S_u$ and the corresponding circle
$M_\gamma(\set{s})$ an interface of $(\gamma,u)$.  The constant
$\sigma$ is given by
\begin{equation*}
  \sigma
  =
  2 \int_{-1}^{1} \sqrt{W(u)} \,d u,
\end{equation*}
and
\begin{equation*}
  \calH^1(M_\gamma(S_u)) = 2\pi \sum_{s \in S_u} y(s)
\end{equation*}
is the one-dimensional Hausdorff measure of the union of the countably
many circles $M_\gamma(S_u)$.

The difference between $\calC_\eps$ and $\calC$ is that $\gamma \in
\calC$ may touch the axis of revolution also in the interior of $I$,
but this can happen only at finitely many points. For $\gamma \in
\calC$ we infer from Lemma~\ref{lem:surf:sec-fform-bound} and the
subsequent corollary that $M_\gamma$ consists of finitely many
components which are $C^1$-surfaces and $W^{2,2}$-surfaces away from
the axis of revolution.

The set $\calP$ resembles the set of special functions of bounded
variation SBV with values in $\set{\pm1}$, weighted with the height
$y$ of the generating curve $\gamma \in \calC$. Indeed, for $u \in
\calP$ and any $J \Subset \set{y>0}$ we have $u \in SBV(J;
\set{\pm1})$, but as jumps of height $2$ may accumulate near the axis
of revolution, $u \not\in SBV(I)$ in general. Points in $\set{y=0}$ can
be jump points of $u$ or singular points where one or both one-sided
limits are undefined. We emphasise that in our notation $S_u$ only
contains points in $\set{y>0}$, because the restriction of $u$ to
$\set{y=0}$ does not contribute to the energy $\calE$.


\subsection{\texorpdfstring{$\Gamma$}{\textGamma}-convergence}

We extend $\calE_\eps$ and $\calE$ to $W^{1,1}(I;\bbR^2) \times
L^1(I)$ by setting $\calE_\eps(\gamma,u) = \calE(\gamma,u) = \infty$
whenever $(\gamma,u)$ does not belong to $\calC_\eps \times
\calP_\eps$ and $\calC \times \calP$, respectively. Our approximation
theorem is the following.

\begin{theorem}
  \label{thm:gamma-conv}
  The energies $\calE_\eps$ are equi-coercive, that is, any sequence
  $(\gamma_\eps,u_\eps) \in \calC_\eps \times \calP_\eps$ with
  uniformly bounded energy admits a subsequence that converges
  strongly in $W^{1,1}(I;\bbR^2) \times L^1(I)$ to some $(\gamma,u)
  \in \calC \times \calP$.
  Furthermore, $\calE_\eps$ $\Gamma$-converges to $\calE$ as $\eps \to
  0$, that is,
  \begin{itemize}
  \item for any sequence $(\gamma_\eps,u_\eps)$ that converges to some
    $(\gamma,u)$ in $W^{1,1}(I;\bbR^2) \times L^1(I)$ we have
    \begin{equation*}
      \liminf_{\eps \to 0} \calE_\eps(\gamma_\eps,u_\eps)
      \geq
      \calE(\gamma,u);
    \end{equation*}
  \item for any $(\gamma,u)$ with finite energy $\calE(\gamma,u)$
    there is a recovery sequence $(\gamma_\eps,u_\eps)$ that converges
    to $(\gamma,u)$ in $W^{1,1}(I;\bbR^2) \times L^1(I)$ and satisfies
    \begin{equation*}
      \limsup_{\eps\to0} \calE_\eps(\gamma_\eps,u_\eps)
      \leq
      \calE(\gamma,u).
    \end{equation*}
  \end{itemize}
\end{theorem}

\begin{remark}[Existence of minimisers]
  The energy $\calE_\eps$ is bounded from below on $\calC_\eps \times
  \calP_\eps$; thus, there is a sequence $(\gamma_\eps,u_\eps)$ such
  that $\calE_\eps(\gamma_\eps,u_\eps) = \inf \calE_\eps +
  o(1)_{\eps\to0}$. From equi-coercivity and $\Gamma$-convergence we
  infer that a subsequence of $(\gamma_\eps,u_\eps)$ converges to a
  minimiser of $\calE$ in $\calC \times \calP$, whose existence is
  thus established; see for instance \cite{Braides02} for the
  details.
\end{remark}

\begin{figure}
  \centering
  \includegraphics[width=.95\linewidth]{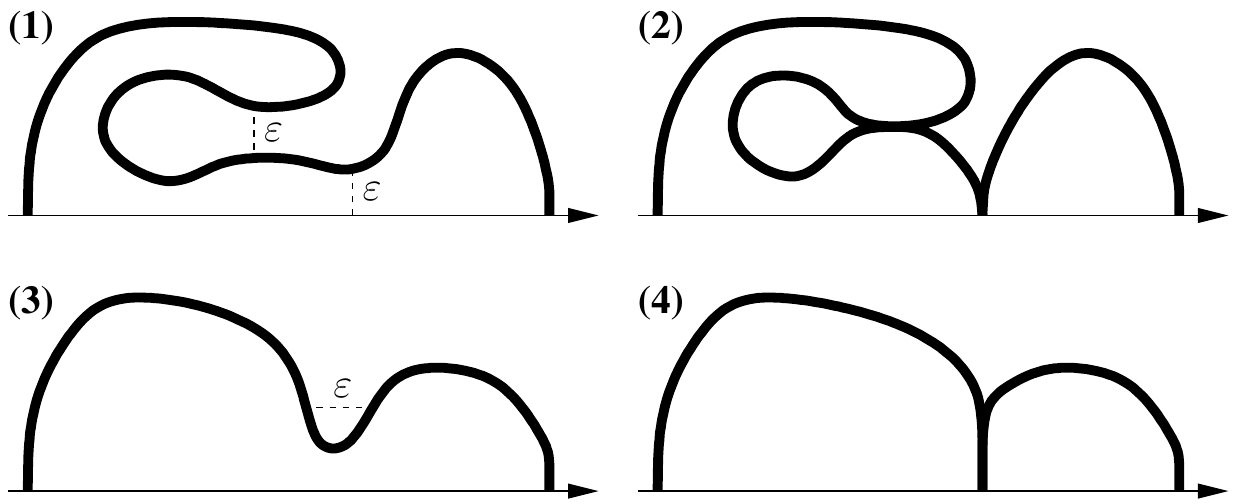}
  \caption{Examples of embedded curves (1) and (3) in $\calC_\eps$
    that lead to non-embedded limits (2) and (4), respectively, in
    $\calC$.  The curves in (3) and (4) satisfy even the stronger
    condition $x' \geq 0$ which prevents a component from touching
    itself, but different components may still touch each other in
    vertical segments near the axis of revolution.}
  \label{fig:embedded-curves}
\end{figure}

\begin{remark}[Embeddedness]
  Our setting and result, which are entirely based on
  parametrisations, do neither need nor guarantee embeddedness. Even
  if $\calE_\eps$ is considered only on the subset of embedded
  membranes or for curves $\gamma=(x,y)$ that satisfy the stronger
  ``generalised graph'' condition $x'\geq0$, which is preserved under
  our convergence, limit curves can touch themselves; see
  Figure~\ref{fig:embedded-curves} for two examples.
\end{remark}


\section{Proof of Theorem \ref{thm:gamma-conv}}
\label{sec:proof}

The proof of Theorem \ref{thm:gamma-conv} is divided into the three
steps equi-coercivity, lower bound and upper bound inequality.


\subsection{Equi-coercivity}

Recalling \eqref{eq:thm:cond1}, which states
\begin{equation*}
  \calE_\eps(\gamma,u)
  \geq
  C_1 \int_{M_\gamma} H^2 \,d\mu - C
\end{equation*}
for some constants $C_1,C>0$ independent of $(\gamma,u) \in \calC_\eps
\times \calP_\eps$, and adding $-8\pi C_1 = -C_1\int_{M_\gamma} 2K
\,d\mu$ to both sides, we find
\begin{equation*}
  \calE_\eps(\gamma,u) + C
  \geq
  C_1 \int_{M_\gamma} |B|^2 \,d\mu
  \geq
  C_1 \int_{M_\gamma} |K| \,d\mu.
\end{equation*}
This means, that $\calE_\eps(\gamma,u)$ bounds the $L^2$-norms of $B$
and $H$ as well as the $L^1$-norm of the Gauss curvature of
$M_\gamma$. Since moreover $\calE_\eps(\gamma,u)+C$ also bounds the
phase field energy from above, compactness for curves and phase fields
can be established separately.


\begin{lemma}
  \label{lem:coerc-curves}
  Let $(\gamma_\eps,u_\eps) \in \calC_\eps \times \calP_\eps$ be a
  sequence with uniformly bounded energy
  $\calE_\eps(\gamma_\eps,u_\eps)$.
  Then there are $\gamma=(x,y) \in \calC$ and a subsequence, not
  relabelled, such that
  \begin{itemize}
  \item $\gamma_\eps \wsto \gamma$ in $W^{1,\infty}(I;\bbR^2)$;
  \item $\gamma_\eps \wto \gamma$ in $W^{2,2}_{\loc}(\set{y>0};
    \bbR^2)$; and
  \item $\gamma_\eps \to \gamma$ in $W^{1,p}(I;\bbR^2)$ for any
    $p \in [1,\infty)$.
  \end{itemize}
\end{lemma}

\begin{proof}
  Let $\gamma_\eps = (x_\eps,y_\eps)$ and $|\gamma'_\eps|=q_\eps$.
  Using Lemma \ref{lem:surf:length-bound} and H{\"o}lder's inequality
  we find
  \begin{equation*}
    2\pi q_\eps |I|
    =
    2\pi \calL_\eps
    \leq
    \int_{M_\eps} |H_\eps| \,d\mu_\eps
    \leq
    \left(
      \calA_\eps \int_{M_\eps} H_\eps^2 \,d\mu_\eps
    \right)^{1/2},
  \end{equation*}
  which bounds the sequence $(q_\eps)$ from above. Furthermore,
  translations in $x$-direction do not change the energy, so we may
  assume that all $\gamma_\eps$ have a common end point and conclude
  that $(\gamma_\eps)$ is bounded in $W^{1,\infty}(I; \bbR^2)$.  We
  can therefore extract a subsequence such that $q_\eps \to q$ in
  $\bbR$ and $\gamma_\eps \wsto \gamma$ in $W^{1,\infty}(I; \bbR^2) =
  C^{0,1}(I; \bbR^2)$; by compact embedding of $W^{1,\infty}$ into
  $C^0$, the convergence of $\gamma_\eps$ is uniform in
  $\overline{I}$. This clearly implies $y \geq 0$ and $y(\partial I)
  = \set{0}$, but also $q>0$ and $y \not\equiv 0$ because
  \begin{equation*}
    A_0
    =
    \calA_\eps
    =
    2\pi q_\eps \int_{I} y_\eps \,d t
    \to
    2\pi q \int_{I} y \,d t.
  \end{equation*}
  Without loss of generality we assume $q=1$, thus $|\gamma'| \leq 1$
  almost everywhere in $I$.

  Taking into account only the just selected subsequence, let $\eps$
  be sufficiently small so that $q_\eps \leq 2$, and let $J \Subset
  \set{y>0}$ and $c_J>0$ be such that $y \geq 2 c_J$ in $J$. By
  uniform convergence of $y_\eps$ we have $y_\eps \geq c_J$ for all
  small $\eps$, and \eqref{eq:surf:kappa1-and-gamma} yields
  \begin{equation}
    \label{eq:W22_loc-bound}
    \frac{1}{2\pi} \int_{M_{\eps}} |B_\eps|^2 \,d\mu_\eps
    \geq
    \frac{1}{2\pi} \int_{M_{\eps}(J)} \kappa_{1,\eps}^2 \,d\mu_\eps
    \geq
    \frac{c_J}{8} \int_{J} |\gamma_\eps''|^2 \,d t.
  \end{equation}
  Since the left hand side of~\eqref{eq:W22_loc-bound} is uniformly
  bounded, a subsequence of $(\gamma_\eps'')$ converges weakly in
  $L^2(J; \bbR^2)$ to some $\gamma_J''$. For this subsequence,
  $\gamma_\eps$ converges weakly in $W^{2,2}(J; \bbR^2)$, and from
  uniqueness of the weak limit we infer that $\gamma_{J}''$ is the
  weak derivative of $\gamma'$ in $J$ and that the whole sequence
  converges. This proves $\gamma_\eps \wto \gamma$ in
  $W^{2,2}_{\loc}(\set{y>0}; \bbR^2)$, and we obtain
  \begin{equation}
    \label{eq:B-local-liminf}
    \int_{M_\gamma(J)} |B|^2 \,d\mu
    \leq
    \liminf_{\eps\to0} \int_{M_\eps(J)} |B_\eps|^2 \,d\mu_\eps
    \leq
    \liminf_{\eps\to0} \int_{M_\eps} |B_\eps|^2 \,d\mu_\eps
  \end{equation}
  for any $J \Subset \set{y>0}$. Exhausting $\set{y>0}$ by $J
  \Subset \set{y>0}$, we conclude
  \begin{equation*}
    \int_{M_\gamma(\set{y>0})} |B|^2 \,d\mu < \infty,
  \end{equation*}
  because the right hand side of \eqref{eq:B-local-liminf} is finite
  and independent of $J$.

  From the compact embedding of $W^{2,2}$ into $C^1$ we know that
  $\gamma_\eps$ converges strongly to $\gamma$ in
  $C^1_{\loc}(\set{y>0}; \bbR^2)$, which implies $\gamma_\eps' \to
  \gamma'$ pointwise in $\set{y>0}$. Thus, we find $|\gamma'| =
  \lim |\gamma_\eps'| = \lim q_\eps = 1$ in $\set{y>0}$ and
  \begin{equation*}
    \calA_\gamma
    =
    2\pi \int_I |\gamma'| y \,d t
    =
    2\pi \int_{\set{y>0}} y \,d t
    =
    \lim_{\eps \to 0}
    2\pi \int_{\set{y>0}} q_\eps y_\eps \,d t
    =
    \lim_{\eps \to 0}
    \calA_\eps
    =
    A_0.
  \end{equation*}

  Finally, to conclude $\gamma \in \calC$ we have to show that
  $\set{y=0}$ is finite. This also yields strong convergence in
  $W^{1,p}(I; \bbR^2)$, because it implies $\gamma_\eps' \to \gamma'$
  almost everywhere in $I$. Assume for contradiction that $J$ is a
  non-empty open subset of $\set{y=0}$. From
  \begin{equation*}
    \left( \int_J |x_\eps'| \,d t \right)^2
    =
    \left( \int_J \frac{|x_\eps'|}{\sqrt{q_\eps y_\eps}}
      \sqrt{q_\eps y_\eps} \,d t \right)^2
    \leq
    \frac{\calA_\eps(J)}{4\pi^2} \int_{M_\eps} \kappa_{2,\eps}^2 \,d\mu_\eps
  \end{equation*}
  we then see that $x_\eps' \to 0$ and $y_\eps'^2 = q_\eps^2 -
  x_\eps'^2 \to 1$ in $L^1(J)$, which contradicts $y'=0$ almost
  everywhere in $\set{y=0}$. Consequently, $\set{y=0}$ does not
  contain interior points, and since by
  Lemma~\ref{lem:surf:sec-fform-bound} the number of components of
  $\set{y>0}$ is finite, we conclude $\calH^0(\set{y=0}) < \infty$.
\end{proof}

\begin{lemma}
  \label{lem:coerc-phase-fields}
  Let $(\gamma_\eps,u_\eps) \in \calC_\eps \times \calP_\eps$ and
  $\gamma \in \calC$ be as in Lemma \ref{lem:coerc-curves}.
  Then there exist a countable set $S \subset I$ with $S \cap J$
  finite for any $J \Subset \set{y>0}$ and $u \in \calP$ with $S_u
  \subset S$ such that for a subsequence $u_\eps \to u$ in measure,
  almost everywhere in $I$, and in $L^p(I)$ for $p \in [1,\infty)$.
\end{lemma}

\begin{proof}
  We restrict ourselves to a subsequence of $\gamma_\eps$ that
  converges to $\gamma$ according to Lemma \ref{lem:coerc-curves}; as
  above, we let $|\gamma_\eps'| \equiv q_\eps$ and without loss of
  generality $|\gamma'| \equiv 1$. Uniform convergence implies that
  for $J \Subset \set{y>0}$ there is $c_J>0$ such that $y_\eps \geq
  c_J$ in $J$ for all sufficiently small $\eps$. Therefore, we have
  \begin{equation}
    \label{eq:coerc-phase-fields}
    \frac{1}{2\pi}
    \int_{M_\eps(J)} \eps |\snabla[M_\eps]{u_\eps}|^2 + \frac{1}{\eps}
    W(u_\eps) \,d\mu_\eps
    \geq
    c_J \int_{J} \frac{\eps}{q_\eps} |u_{\eps}'|^2 +
    \frac{q_\eps}{\eps} W(u_{\eps}) \,d t
  \end{equation}
  and the well-known arguments of Modica and Mortola
  \cite{Modica87,MoMo77} apply in $J$; see also \cite[Lemma 6.2 and
  Remark 6.3]{Braides02} for a proof in one dimension. The outcome is
  a finite set of points $S_J \subset J$ and a piecewise constant
  function $u \colon J \to \set{\pm1}$ whose jump set is contained in
  $S_J$ such that a subsequence of $u_\eps$ converges to $u$ in
  measure and almost everywhere in $J \sm S$. Since $(u_\eps)$ is
  uniformly bounded in $L^\infty(I)$, convergence in $L^p(I)$ for any
  $p<\infty$ follows.

  Exhausting $\set{y>0}$ by a sequence of increasing sets such as $J_k
  = \set{y>1/k}$ for $k \to \infty$ and taking a diagonal sequence, we
  find an at most countable set $S \subset \set{y>0}$ and a function
  $u \colon \set{y>0} \to \set{\pm1}$ whose jump set is contained in
  $S$. Moreover, a subsequence of $(u_\eps)$ converges to $u$ in
  measure and almost everywhere in $\set{y>0}$. Then
  $\calH^0(\set{y=0}) < \infty$ and $\|u_\eps\|_\infty \leq C_0$
  provide convergence in $L^p(I)$ for any $1 \leq p < \infty$, and
  taking convergence of $y_\eps$ and $|\gamma_\eps'|$ into account, we
  obtain
    \begin{equation*}
    m A_0
    =
    \int_{M_{\eps}} u_\eps \,d\mu_\eps
    \to
    \int_{M_\gamma} u \,d\mu
  \end{equation*}
  as $\eps \to 0$. The bound $\calH^1(M_\gamma(S_u)) < \infty$
  follows from \eqref{eq:coerc-phase-fields} and Young's inequality;
  the details are given in the lower bound section and are thus here
  omitted.
\end{proof}

\begin{remark}
  In the classical one-dimensional setting without the area measure, a
  uniform $L^\infty$-bound for the phase fields is in fact a result of
  the uniform energy bound; see \cite{Braides02}. In our case,
  however, this bound depends in $J \Subset \set{y>0}$ on the constant
  $c_J$, which is essentially the infimum of $y$ on $J$, and tends to
  infinity as $c_J \to 0$.
\end{remark}


\subsection{Lower bound}

Next we prove the lower bound inequality
\begin{equation}
  \label{eq:lower-bound}
  \liminf_{\eps \to 0} \calE_{\eps}(\gamma_\eps,u_\eps)
  \geq
  \calE(\gamma,u)
\end{equation}
whenever $(\gamma_\eps,u_\eps)$ converges to $(\gamma,u)$ in
$W^{1,1}(I;\bbR^2) \times L^1(I)$. It suffices to examine the case
when the left hand side of \eqref{eq:lower-bound} is finite and to
consider a subsequence such that the lower limit is attained. Then by
definition $(\gamma_\eps,u_\eps) \in \calC_\eps \times \calP_\eps$,
and our compactness argument yields $(\gamma,u) \in \calC \times
\calP$ and the convergence properties listed in Lemmas
\ref{lem:coerc-curves} and \ref{lem:coerc-phase-fields}.

Recalling the formulas $\kappa_{1,\eps} = - \gamma_\eps'' \cdot
\gamma_\eps'^\perp / q_\eps^3$ and $\kappa_{2,\eps} = x_\eps'/(q_\eps
y_\eps)$ for the principal curvatures, we find that $\gamma_\eps \wto
\gamma$ in $W^{2,2}_{\loc}(\set{y>0};\bbR^2)$ implies weak convergence
of $H_\eps$ and $K_\eps$ in $L^2_\loc(\set{y>0})$. Together with
$y_\eps q_\eps \to y q$ uniformly, $u_\eps \to u$ in $L^1(I)$, and the
$L^\infty$-bounds for $k$ and $k_G$ this yields
\begin{align*}
  \int_{M_\gamma(J)} k(u) (H-H_{s}(u))^2 &+ k_G(u) K \,d\mu
  \\
  &\leq
  \liminf_{\eps\to0} \int_{M_\eps(J)}
  k(u_\eps) (H_\eps-H_{s}(u_\eps))^2 + k_G(u_\eps) K_\eps \,d\mu_\eps
\end{align*}
for any $J \Subset \set{y>0}$. Adding temporarily $C A_0$, where $C$
is the constant in \eqref{eq:thm:cond2}, to make the integral on the
right hand side non-negative, we estimate the latter by extending it
to the whole surface $M_\eps$ and exhaust $\set{y>0}$ by intervals $J
\Subset \set{y>0}$ on the left hand side. Thereby, we obtain the bulk
lower bound
\begin{equation}
  \begin{aligned}
    \label{eq:bulk-lower-bound}
    \int_{M_\gamma(\set{y>0})} k(u) (H-H_{s}(u))^2
    &+ k_G(u) K \,d\mu
    \\
    &\leq
    \liminf_{\eps\to0}
    \int_{M_\eps} k(u_\eps) (H_\eps-H_{s}(u_\eps))^2
    + k_G(u_\eps) K_\eps \,d\mu_\eps.
  \end{aligned}
\end{equation}

To analyse the interface energy let $s \in S_u$ and fix an interval $J
\Subset \set{y>0}$ such that $\overline{J} \cap S_u = \set{s}$, which
exists because $S_u \cap \set{y > y(s)/2}$ is finite. From the
convergence of $u_\eps$ we deduce that there are points $a_\eps,
b_\eps \in J$ with $a_\eps < s < b_\eps$ or $b_\eps < s < a_\eps$ such
that $a_\eps \to s$, $b_\eps \to s$, $u_\eps(a_\eps) \to -1$, and
$u_\eps(b_\eps) \to 1$ as $\eps \to 0$. Assuming without loss of
generality that $a_\eps<b_\eps$, we have
\begin{align*}
  \frac{1}{2\pi} \int_{M_\eps(a_\eps,b_\eps)}
  \eps |\snabla[M_\eps]{u_\eps}|^2 + \frac{1}{\eps} W(u_\eps) \,d\mu_\eps
  &\geq
  \left( \inf_{(a_\eps,b_\eps)} y_\eps \right)
  \int_{a_\eps}^{b_\eps} 2 \sqrt{W(u_\eps)}|u_\eps'| \,d t
  \\
  &\geq
  \left( \inf_{(a_\eps,b_\eps)} y_\eps \right)
  \left| \int_{u_\eps(a_\eps)}^{u_\eps(b_\eps)} 2 \sqrt{W(u)} \,d u \right|
\end{align*}
thanks to Young's inequality and a change of variables. Taking the
lower limit yields
\begin{equation}
  \label{eq:single-interf-lower-bound}
  \liminf_{\eps \to 0}
  \int_{M_\eps(a_\eps,b_\eps)}
  \eps |\nabla_{\! {M_\eps}} u_\eps|^2 + \frac{1}{\eps} W(u_\eps) \,d\mu_\eps
  \geq
  2 \pi y(s) \int_{-1}^{1} 2 \sqrt{W(u)} \,d u
  =
  2 \pi y(s) \sigma.
\end{equation}
The above argument applies to each point of any finite subset $S$ of
$S_u$, and in addition we may extend the integral on the left hand
side of \eqref{eq:single-interf-lower-bound} to the whole surface
to obtain
\begin{equation*}
  \liminf_{\eps \to 0}
  \int_{M_\eps} 
  \eps |\nabla_{\! {M_\eps}} u_\eps|^2 + \frac{1}{\eps} W(u_\eps) \,d\mu_\eps
  \geq
  \sigma \calH^1(M_\gamma(S)).
\end{equation*}
Since the left hand side is independent of $S$, the interface lower
bound inequality
\begin{equation}
  \label{eq:inter-lower-bound}
  \liminf_{\eps \to 0}
  \int_{M_\eps} 
  \eps |\nabla_{\! {M_\eps}} u_\eps|^2 + \frac{1}{\eps} W(u_\eps) \,d\mu_\eps
  \geq
  \sigma \calH^1(M_\gamma(S_u))
\end{equation}
follows from taking the supremum over all finite sets $S \subset S_u$.
Combining \eqref{eq:inter-lower-bound} and \eqref{eq:bulk-lower-bound}
yields the lower bound inequality \eqref{eq:lower-bound}.


\subsection{Upper bound}

We now construct a recovery sequence for $(\gamma,u)$ with finite
energy $\calE(\gamma,u)$. To this end, we first show that $(\gamma,u)$
can be approximated by membranes with finitely many interfaces. For
such a membrane, we then obtain a recovery sequence by changing the
curve essentially only near component boundaries and the phase field
only around interfaces and component boundaries. Finally, a diagonal
sequence recovers $(\gamma,u)$.

Throughout this section we assume without loss of generality that
$|\gamma'| \equiv 1$.


\subsubsection{Approximation by finite number of interfaces}

\begin{lemma}
  \label{lem:approx-noi-finite}
  Assume that $(\gamma,u) \in \calC \times \calP$ has countably many
  interfaces. Then there exists $(\gamma,u_\delta) \in \calC \times
  \calP$ for sufficiently small $\delta>0$, each with a finite number
  of interfaces, such that $u_\delta \to u$ in $L^p(I)$ for any $p \in
  [1,\infty)$ and $\calE(\gamma,u_\delta) \to \calE(\gamma,u)$ as
  $\delta \to 0$.
\end{lemma}

\begin{proof}
  Let $\gamma=(x,y)$ and $3\delta$ be smaller than the minimal length
  of a component of $\{y>0\}$. We construct $u_\delta$ by omitting
  interfaces whose distance on $\gamma$ to a component boundary is
  less than $\delta$. More precisely, for a component $\omega = (a,b)$
  of $\set{y>0}$ we let $a_\delta = a + \delta$ and $b_\delta = b -
  \delta$, which implies $a_\delta < b_\delta$ and
  $\calL_\gamma(a,a_\delta) = \calL_\gamma(b_\delta,b) = \delta$, and
  define $u_\delta$ on $\omega$ to be the continuous extension of $u$
  from $(a_\delta,b_\delta)$ to $\omega$, that is,
  \begin{equation*}
    u_\delta =
    \begin{cases}
      u
      &\text{in } (a_\delta,b_\delta),
      \\
      \lim\limits_{t \searrow a_\delta} u(t)
      & \text{in } (a,a_\delta],
      \\
      \lim\limits_{t \nearrow b_\delta} u(t)
      & \text{in } [b_\delta,b).
    \end{cases}
  \end{equation*}

  Since the number of components $N_\gamma$ is finite, this can be
  done separately for each component, and the composition yields a
  membrane $(\gamma,u_\delta)$ with finitely many interfaces. By
  construction, we have $|u-u_\delta| \leq 2$ and $y \leq \delta$ in
  $(a,a_\delta) \cup (b_\delta,b)$, so we find
  \begin{equation}
    \label{eq:phase-field-error}
    \int_{M_\gamma} |u-u_\delta|^p \,d\mu
    \leq
    2^{p+1} N_\gamma \delta^2
  \end{equation}
  and $u_\delta \to u$ as $\delta\to0$ in $L^p(I)$ for any $p \in
  [1,\infty)$. Furthermore,
  \begin{equation*}
    | \calH^1(M_\gamma(S_u)) - \calH^1(M_\gamma(S_{u_\delta})) |
    \leq
    \calH^1(M(S_u \cap \set{y \leq \delta}))
  \end{equation*}
  as well as
  \begin{multline*}
    \left|
      \int_{M_\gamma}
      k(u) (H-H_{s}(u))^2 + k_G(u) K
      - k(u_\delta) (H-H_{s}(u_\delta))^2 - k_G(u_\delta) K \,d\mu
    \right|
    \\
    \leq
    \int_{M_\gamma(\set{y \leq \delta})}
    2 \|k\|_\infty H^2
    + 4 \|k H_s\|_\infty \left( |H| + \|H_s\|_\infty \right)
    + 2 \|k_G\|_\infty |K| \,d\mu
  \end{multline*}
  vanish in the limit $\delta\to0$, and we deduce
  $\calE(\gamma,u_\delta) \to \calE(\gamma,u)$.

  Finally, for sufficiently small $\delta$ there is an interface $s
  \in S_u \cap S_{u_\delta}$ that is independent of $\delta$ and whose
  distance to all other interfaces is greater than $\delta$. According
  to \eqref{eq:phase-field-error} the error in the phase constraint is
  at most of order $\delta^2$, so it suffices to move $s$ by an order
  of at most $\delta^2$ to the left or right to recover the integral
  constraint $\int_M u_\delta \,d\mu = m A_0$. This additional change
  yields $u_\delta \in \calP$ and does not disturb the convergence of
  phase fields and energy.
\end{proof}

In virtue of Lemma \ref{lem:approx-noi-finite} we assume from now on
that $(\gamma,u)$ has only finitely many interfaces. Then $u$ is
either continuous at points in $\set{y=0}$ or has a well-defined jump.
Moreover, the minimal distance between two interfaces and from an
interface to the boundary of its component is positive. Hence, for any
interface $s \in S_u$ there is an interval $J \Subset \set{y>0}$ that
contains $s$ but no other interface, and for any component boundary
point $s \in \set{y=0} \sm \partial I$ there is an interval $J \subset
I$ that contains $s$ but no other component boundary or interface.


\subsubsection{Local interface recovery}
\label{sec:local-interf-recov}

The recovery of a phase field $u$ with finitely many jumps follows the
lines of the Modica-Mortola theory for phase transitions. The main
difference is the inhomogeneity due to the area measure $d\mu = 2\pi y
\,d t$, but since $u$ will be changed only in an interval of order
$\sqrt{\eps}$ around each interface, this issue is easily dealt with.

It is well known, see for instance~\cite{Alberti00}, that in the
classical one-dimensional setting the $\eps$-energy-minimal profile
for a transition of $u_\eps$ from $-1$ to $+1$ is obtained by
minimising
\begin{equation*}
  G_\eps(u) = \int_{\bbR} \eps |u'|^2 + \frac{1}{\eps} W(u) \,d t
\end{equation*}
among functions $u$ that satisfy $u(0)=0$ and $u(\pm\infty) = \pm
1$. Indeed, setting $u_\eps(t) = u(t/\eps)$ we observe
\begin{equation*}
  G_\eps(u_\eps) = G_1(u)
  \geq
  2\int_{\bbR} \sqrt{W(u)}u' \,d t
  =
  2 \int_{\bbR} \sqrt{W(u)} \,d u
  =
  \sigma,
\end{equation*}
and equality holds if and only if
\begin{equation}
  \label{eq:phase-field-equation}
  u'=\sqrt{W(u)}.
\end{equation}
Equation \eqref{eq:phase-field-equation} admits a local solution $p$
with initial condition $p(0)=0$, because $\sqrt{W(\cdot)}$ is
continuous. Since the constants $+1$ and $-1$ are a global super- and
sub-solution of \eqref{eq:phase-field-equation}, $p$ can be extended
to the whole real line, and due to $W(p)>0$ for $p\in(-1,+1)$, we
obtain $p(t) \to \pm1$ as $t\to\pm\infty$. As a consequence,
$p(t/\eps)$ is admissible and minimises $G_\eps$. Furthermore, by
symmetry of $W$ we can presume $-p(-t)=p(t)$ and need to know the
profile only for $t\geq0$.

Let $(\gamma,u) \in \calC \times \calP$ have finitely many interfaces
and consider $s \in S_u$ and $J \Subset \set{y>0}$ such that
$\overline{J} \cap S_u = \set{s}$. For simplicity of notation we
assume $s=0$. Using an appropriately scaled version of the optimal
profile $p$ and a linear interpolation, we aim to construct the
recovery sequence by replacing $u = \sign t$ on $J$ with
\begin{equation*}
  p_\eps(t) =
  \begin{cases}
    p(t/\eps)
    &\text{if }
    0 \leq t < \sqrt{\eps},
    \\
    p(1/\sqrt{\eps}) + \frac{1}{\eps}(t - \sqrt{\eps})
    &\text{if }
    \sqrt{\eps} \leq t < \sqrt{\eps} + \eps (1-p(1/\sqrt{\eps})),
    \\
    1
    &\text{if }
    \sqrt{\eps} + \eps (1-p(1/\sqrt{\eps})) \leq t
  \end{cases}
\end{equation*}
for $t \geq 0$ and $p_\eps(t)=-p_\eps(-t)$ for $t<0$; if $u = - \sign
t$ in $J$, we use $-p_\eps$. Since $\gamma$ is in general not
symmetric around $s=0$ we have to correct $p_\eps$ in order to
conserve the phase integral constraint.

\begin{lemma}
  \label{lem:phase-field-limsup}
  There is $u_\eps \in W^{1,2}(J)$ with $\set{u_\eps \not= u} \Subset
  J$ such that $\|u_\eps\|_\infty \leq C_0$, $u_\eps \to u$ in
  $L^1(J)$, $\int_{M(J)} u_\eps \,d\mu = \int_{M(J)} u \,d\mu$, and
  \begin{equation}
    \label{eq:phase-field-limsup}
    \limsup_{\eps \to 0}
    \int_{M_\gamma(J)}
    \eps |\snabla[M_\gamma]{u_\eps}|^2 + \frac{1}{\eps} W(u_\eps) \,d\mu
    \leq
    2 \pi \sigma y(s).
  \end{equation}
\end{lemma}

\begin{proof}
  Convergence $p_\eps \to u$ in $L^1(J)$ is obvious from the
  definition of $p_\eps$, and the estimate
  \eqref{eq:phase-field-limsup} with $p_\eps$ instead of $u_\eps$
  follows by taking the upper limit $\eps \to 0$ in
  \begin{align*}
    \int_{M_\gamma(J)}
    \eps |\snabla[M_\gamma]{p_\eps}|^2 + \frac{1}{\eps} W(p_\eps) \,d\mu
    &\leq
    2\pi
    \left( \sup_{[-\eps,\eps]} y \right)
    \int_{-1\sqrt{\eps}}^{1/\sqrt{\eps}} |p'(t)|^2 + W(p(t)) \,d t
    \\
    &\quad
    + 2\pi
    \left( 1-p(1/\sqrt{\eps}) \right)
    \left( \sup_{J} y \right)
    \left( 1+\sup_{[-1,1]} W \right).
  \end{align*}
  To recover the constraint, let $f \colon J \to \bbR$ be smooth, have
  compact support in $J \cap \set{t>0}$ and satisfy $\int_{M_\gamma(J)} f
  \,d\mu = 1$. Then the phase integral is conserved by $u_\eps =
  p_\eps + \alpha_\eps f$ if
  \begin{equation*}
    \alpha_\eps = \int_{M_\gamma(J)} u-p_\eps \,d\mu.
  \end{equation*}
  From
  \begin{equation*}
    \int_{M_\gamma(0,\sqrt{\eps})} 1-p_\eps \,d\mu
    \leq
    2\pi \|y\|_\infty \sqrt{\eps} \int_0^1 1-p(t/\sqrt{\eps}) \,d t
    =
    o(\sqrt{\eps}),
  \end{equation*}
  we infer that $\alpha_\eps$ is of order $o(\sqrt{\eps})$, which is
  sufficient to ensure convergence $u_\eps \to u$ in $L^1(J)$ and the
  energy inequality
  \begin{equation*}
    \limsup_{\eps \to 0}
    \int_{M_\gamma(J)}
    \eps |\snabla[M_\gamma]{u_\eps}|^2 + \frac{1}{\eps} W(u_\eps) \,d\mu
    \leq
    2 \pi \sigma y(s)
  \end{equation*}
  thanks to
  \begin{equation*}
    \frac{1}{\eps} W(\pm1+\alpha_\eps f)
    =
    \frac{1}{\eps} \left( 
      W(\pm1) + \alpha_\eps f W'(\pm1) + O(\alpha_\eps^2)
    \right)
    =
    o(1).
  \end{equation*}
  By construction, we have $u_\eps \in W^{1,2}(J)$ and
  $\|u_\eps\|_\infty \leq \|p_\eps\|_\infty + |\alpha_\eps| 
  \|f\|_\infty \leq C_0$ for all sufficiently small $\eps>0$.
\end{proof}

\begin{remark}
  Lemma \ref{lem:phase-field-limsup} remains true if $\gamma$ is
  replaced by a sequence $\gamma_\eps$ that satisfies
  $|\gamma_\eps'|\equiv q_\eps \to 1$,
  $|\calA_\eps(J)-\calA_\gamma(J)|=o(\sqrt{\eps})$ and $\gamma_\eps
  \to \gamma$ in $W^{1,p}(J;\bbR^2)$ for some $p \in [1,\infty)$.
\end{remark}


\subsubsection{Curve approximation and recovery}
\label{sec:curve-appr-recov}

To obtain a recovery sequence for the curves we have to change
segments of $\gamma=(x,y)$ near interior points on the axis of
revolution. Close to the axis the second principal curvature becomes
unbounded unless $x'=0$, therefore we base our construction on scaled
catenoids in order to control the mean curvature integral in the
energy $\calE_\eps$. With the topological changes introduced by this
construction and their effect on the Gauss curvature integral we deal
later by adapting the phase field.

One issue with the above idea is that the catenoids have to make a
$C^1$-connection with the original surface, which even after taking
symmetries into account can display several types of behaviour. For
instance, the generating curve might leave the axis of revolution
turning only in one direction, zig-zagging in $x$-direction, or as a
vertical line segment; see Figure \ref{fig:curve-constructions}. In a
first step we therefore reduce the number of possible situations by
showing that a membrane can be approximated by membranes that only
have vertical line segments near the axis of revolution.

\begin{figure}
  \centering
  \includegraphics[width=.95\linewidth]{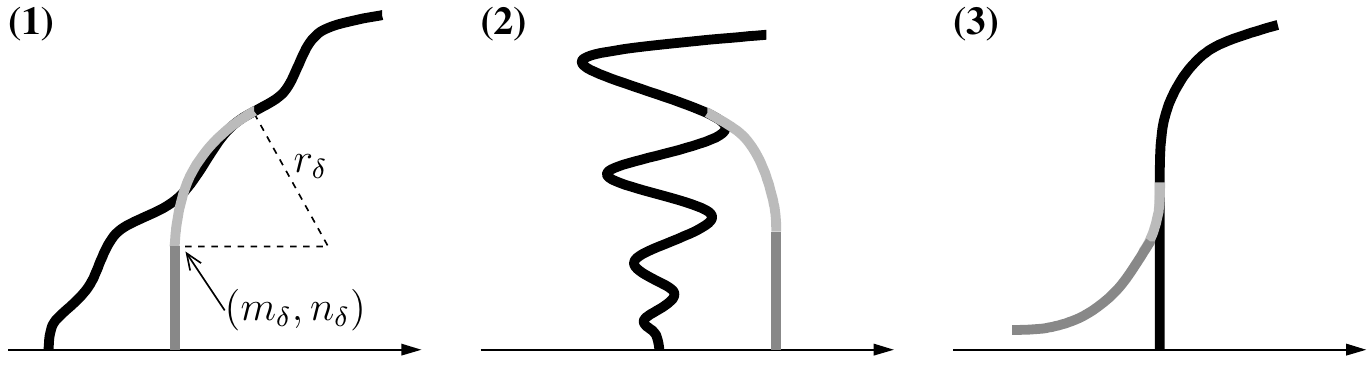}
  \caption{Examples of curves $\gamma$ (black) near the axis of
    revolution and their recovery (grey): $\gamma$ turns in
    one-direction (1), zig-zags (2), or is vertical near the
    axis of revolution (3);
    it is approximated in (1) and (2) by a circle segment (light grey)
    and a vertical line (dark grey) as constructed in Lemma
    \ref{lem:approx-simple-curves};
    the vertical line in (3) is recovered by a catenary (dark grey)
    and a circle segment (light grey) as in Lemma
    \ref{lem:catenoid-and-circle}.}
  \label{fig:curve-constructions}
\end{figure}

\begin{lemma}
  \label{lem:approx-simple-curves}
  Let $(\gamma,u) \in \calC \times \calP$ have finitely many
  interfaces. Then for sufficiently small $\delta>0$ there is
  $(\gamma_\delta,u_\delta) \in \calC \times \calP$ with finitely many
  interfaces such that $\gamma_\delta \to \gamma$ in
  $W^{1,p}(I;\bbR^2)$, $u_\delta \to u$ in $L^p(I)$ for any $p \in
  [1,\infty)$, and $\calE(\gamma_\delta,u_\delta) \to \calE(\gamma,u)$
  as $\delta \to 0$. Moreover, each $\gamma_\delta$ meets the axis of
  revolution in vertical line segments, that is, for any $s \in
  \set{y=0}$ there are $a_\delta,b_\delta \in I$,
  $b_\delta<s<a_\delta$ such that $y_\delta$ restricted to
  $(b_\delta,s)$ and $(s,a_\delta)$, respectively, is a vertical line.
\end{lemma}

\begin{proof}
  For the local construction around a point on the axis of revolution
  we consider the left boundary of one component $\omega = (a,b)$ of
  $M_\gamma$ where $\gamma$ is not vertical; at right component
  boundaries a mirrored version applies. For simplicity of notation
  and we assume $a=0$, $\gamma(0) = (0,0)$, and $|\gamma'|=1$ in $I$.

  From Lemma \ref{lem:surf:sec-fform-bound} we know that $x'(0)=0$ and
  $|y'(t)| \to 1$ as $t \searrow 0$. It is sufficient to consider the
  case $y'(t) \to +1$ as $t \searrow 0$, as the construction for
  $y'(t) \to -1$ is obtained by traversing the former one backwards.
  Let $J = (0,t_0) \subset \omega$ be an interval that contains no
  interface of $(\gamma,u)$ and such that $y' \geq 1/2$ in $J$.  Since
  $\gamma$ is not vertical near $a=0$, we find a sequence $(t_\delta)
  \subset J$ such that $t_\delta \searrow 0$, $y(t_\delta) \searrow
  0$, $x'(t_\delta) \searrow 0$ as $\delta \to 0$ and either
  $x'(t_\delta) > 0$ or $x'(t_\delta) < 0$ for any $\delta$. Again it
  suffices to consider the case $x'(t_\delta)>0$, as the other is
  dealt with by a mirrored construction.

  We aim to connect $\gamma$ at $t=t_\delta$ to a circle with unit
  speed parametrisation $(k_\delta,l_\delta)$ given by
  \begin{equation*}
    k_\delta(t)
    =
    r_\delta - r_\delta \cos \left( \frac{t-a_\delta}{r_\delta} \right)
    + m_\delta
    \qquad\text{and}\qquad
    l_\delta(t)
    =
    r_\delta \sin \left( \frac{t-a_\delta}{r_\delta} \right) + n_\delta
  \end{equation*}
  with radius $r_\delta$, shifts $m_\delta$, $n_\delta$, and parameter
  shift $a_\delta$ to be found; see
  Figure~\ref{fig:curve-constructions}(1). At $t = a_\delta$ we have
  $(k_\delta',l_\delta') = (0,1)$, thus the circle can be connected to
  a vertical line segment provided that $n_\delta =
  l_\delta(a_\delta)>0$. At $t=t_\delta$ we have to satisfy the
  conditions
  \begin{align*}
    k_\delta'(t_\delta) &= x'(t_\delta),
    &
    k_\delta(t_\delta) &= x(t_\delta),
    \\
    l_\delta'(t_\delta) &= y'(t_\delta),
    &
    l_\delta(t_\delta) &= y(t_\delta)
  \end{align*}
  in order to match end points and derivatives of $\gamma$ and the
  circle. A short computation shows that
  \begin{equation}
    \label{eq:approx-circ-1}
    \frac{t_\delta-a_\delta}{r_\delta}
    =
    \arctan \frac{x'(t_\delta)}{y'(t_\delta)}
  \end{equation}
  and 
  \begin{equation*}
    x(t_\delta) = r_\delta (1-y'(t_\delta)) + m_\delta,
    \qquad
    y(t_\delta) = r_\delta x'(t_\delta) + n_\delta.
  \end{equation*}
  These equations determine $a_\delta$, $m_\delta$, and $r_\delta$ in
  terms of the given $t_\delta$, $\gamma(t_\delta)$,
  $\gamma'(t_\delta)$, and the still free $n_\delta$; choosing
  $n_\delta = y(t_\delta)/2 > 0$, we obtain
  \begin{equation*}
    r_\delta = \frac{n_\delta}{x'(t_\delta)},
    \quad
    m_\delta = x(t_\delta) - n_\delta
    \frac{1-y'(t_\delta)}{x'(t_\delta)},
    \quad\text{and}\quad
    a_\delta = t_\delta - n_\delta
    \frac{\arctan( x'/y' )(t_\delta)}{x'(t_\delta)}.
  \end{equation*}
  The shift $n_\delta$ tends to $0$ as $\delta \to 0$ by definition,
  and $m_\delta \to 0$ is a consequence of
  \begin{equation*}
    \frac{1-y'(t_\delta)}{x'(t_\delta)}
    =
    \frac{1-y'(t_\delta)}{\sqrt{1-y'(t_\delta)^2}}
    =
    \sqrt{\frac{1-y'(t_\delta)}{1+y'(t_\delta)}}
    \to 0.
  \end{equation*}
  Moreover, $0 \leq \arctan z \leq z$ for $z \geq 0$ and $1/2 \leq
  y'(t_\delta) \leq 1$ imply
  \begin{equation*}
    t_\delta
    \geq
    a_\delta
    \geq
    t_\delta - \frac{y(t_\delta)}{2 y'(t_\delta)}
    \geq
    t_\delta - y(t_\delta)
    \geq
    0,
  \end{equation*}
  hence $a_\delta \to 0$ as $\delta \to 0$ follows.

  With these circles we define a local approximation for $\gamma$ by
  \begin{equation*}
    \gamma_\delta(t) = 
    \begin{cases}
      \gamma(t)
      &\text{if } t_\delta \leq t,
      \\
      (k_\delta(t),l_\delta(t))
      &\text{if } a_\delta \leq t < t_\delta,
      \\
      (m_\delta,t+n_\delta-a_\delta)
      &\text{if } a_\delta-n_\delta \leq t < a_\delta.
    \end{cases}
  \end{equation*}
  Here the third part is a vertical line segment of unit speed that
  connects $(m_\delta,n_\delta)$ at $t=a_\delta$ with $(m_\delta,0)$
  at $t=a_\delta-n_\delta$. Clearly, $\gamma_\delta$ belongs to
  $W^{2,2}((a_\delta-n_\delta,t_0);\bbR^2)$. Since the vertical line
  and the circle segment vanish in the limit $\delta \to 0$, we have
  pointwise convergence of $\gamma_\delta$ and $\gamma_\delta'$ to
  $\gamma$ and $\gamma'$, respectively. Moreover, the area of
  $M_\delta(a_\delta-n_\delta,t_0)$ converges to the area of
  $M(0,t_0)$.

  On the vertical segment both principal curvatures and all curvature
  integrals are zero. On the circle segment we have
  $|\kappa_{1,\delta}| = 1/r_\delta$ and thus
  \begin{align}
    \nonumber
    \frac{1}{2\pi} \int_{M_\delta(a_\delta,t_\delta)}
    \kappa_{1,\delta}^2 \,d\mu_\delta
    &=
    \int_{a_\delta}^{t_\delta} \frac{1}{r_\delta^2}
    \left( r_\delta \sin ((t-a_\delta)/r_\delta) + n_\delta \right) d t
    \\ 
    \label{eq:first-curv-estimate}
    &\leq
    2\frac{t_\delta-a_\delta}{r_\delta} 
  \end{align}
  using $n_\delta/r_\delta = x'(t_\delta) \in (0,1]$; for the second
  principal curvature $\kappa_{2,\delta} = x_\delta'/y_\delta$ we
  compute
  \begin{align}
    \nonumber
    \frac{1}{2\pi} \int_{M_\delta(a_\delta,t_\delta)}
    \kappa_{2,\delta}^2 \,d\mu_\delta
    &=
    \int_{a_\delta}^{t_\delta} \frac{1}{r_\delta}
    \frac{\sin((t-a_\delta)/r_\delta)^2}
    {\sin((t-a_\delta)/r_\delta) + n_\delta/r_\delta} \,d t
    \\
    \label{eq:second-curv-estimate}    
    &\leq
    \frac{t_\delta-a_\delta}{r_\delta}.
  \end{align}
  Due to \eqref{eq:approx-circ-1} and $x'(t_\delta) \to 0$, both
  \eqref{eq:first-curv-estimate} and \eqref{eq:second-curv-estimate}
  tend to $0$ as $\delta \to 0$, hence we obtain
  \begin{equation*}
    \int_{M_\delta(a_\delta-n_\delta,t_0)} |B_\delta|^2 \,d\mu_\delta
    \to
    \int_{M(0,t_0)} |B|^2 \,d\mu.
  \end{equation*}
  As $u$ is constant in $J=(0,t_0)$, we may define $u_\delta(t) =
  u|_J$ for $t \in [a_\delta-n_\delta,t_0]$ and $u_\delta(t) = u(t)$
  for $t \in \omega, t>t_0$. As with the curves, $u_\delta$ converges
  to $u$ pointwise, and we obtain
  \begin{multline*}
    \int_{M_\delta(a_\delta-n_\delta,t_0)}
    k(u_\delta) \left( H_\delta-H_s(u_\delta) \right)^2
    + k_G(u_\delta) K_\delta \,d\mu_\delta
    \\
    \to
    \int_{M(0,t_0)}
    k(u) \left( H-H_s(u) \right)^2
    + k_G(u) K \,d\mu.
  \end{multline*}
  Note in particular, that both Gauss curvature integrals are equal,
  because they depend only on the tangent angle of $\gamma_\delta$ or
  $\gamma$ at $t_\delta$ and $a_\delta-n_\delta$ or $0$, respectively.

  In order to fit the above construction into the neighbouring
  components of $\omega$, $\gamma_\delta$ and the rest of the original
  curve $\gamma$ have to be shifted in $x$ and $t$. These shifts,
  however, vanish as $\delta \to 0$ and thus do not disturb the proved
  convergences. Applying the above procedure to the boundaries of
  each component and gluing together the resulting segments gives a
  membrane $(\gamma_\delta,u_\delta)$ defined on some interval
  $I_\delta$, which converges to $I$ as $\delta \to 0$ in the sense
  that the boundary points converge.

  It remains to correct the area and the phase field constraint as
  well as the parameter interval. For the area constraint we fix $J
  \Subset \set{y>0} \sm S_u$ such that apart from the shifts $\gamma$
  is unchanged in $J$ for all small $\delta$, $x'>0$ or $x'<0$ in
  $\overline{J}$, and such that $M_\gamma(J)$ is not part of a
  catenoid. Such an interval exists, because otherwise $\gamma$
  restricted to any component of $\set{y>0}$ would consist only of
  vertical lines and catenary segments, which is impossible for a
  $C^1$-curve that starts and ends on the $x$-axis. After an
  additional parameter shift of $\gamma_\delta$ we may assume that
  $\gamma_\delta(J) = \gamma(J)$ up to an $x$-shift. Let $f \in
  C_c^\infty(J; \bbR^2)$ and consider the curve $\widetilde
  \gamma_{\delta,\alpha} = \gamma_\delta + \alpha f$, whose
  corresponding surface of revolution has the area
  \begin{equation*}
    \calA_{\widetilde \gamma_{\delta,\alpha}}
    =
    \calA_{\gamma_\delta} + \calA_{\widetilde \gamma_{\delta,\alpha}}(J)
    - \calA_\gamma(J).
  \end{equation*}
  Then the requirement $\calA_\gamma = \calA_{\widetilde
    \gamma_{\delta,\alpha}}$ is equivalent to
  \begin{equation}
    \label{eq:area-recovery-1}
    \calA_{\widetilde \gamma_{\delta,\alpha}}(J) - \calA_\gamma(J)
    =
    \calA_\gamma - \calA_{\gamma_\delta}.
  \end{equation}
  The left hand side of \eqref{eq:area-recovery-1} equals $0$ for
  $\alpha=0$ and depends continuously on $\alpha$; it is strictly
  positive for one sign of $\alpha$ and strictly negative for the
  other, since $\gamma(J)$ is not a catenoid segment and $M_\gamma(J)$
  not stationary for the area. The right hand side of
  \eqref{eq:area-recovery-1} vanishes as $\delta \to 0$, hence for all
  sufficiently small $\delta$ there is an $\alpha_\delta$ such that
  \eqref{eq:area-recovery-1} holds and $a_\delta \to 0$ as $\delta \to
  0$. Thus, gluing together $\widetilde \gamma_{\delta,\alpha_\delta}$
  instead of $\gamma_\delta$ accounts for the area constraint at the
  cost of violating the constant speed requirement. The latter,
  however, is fixed by a global reparametrisation, which also gives a
  membrane defined on $I$. Since $I_\delta \to I$ as $\delta \to 0$
  and the perturbations from the area recovery vanishes with
  $\alpha_\delta \to 0$ in any function space, these
  reparametrisations converge to the identity in $W^{2,2}$ and the
  convergences of curvature integrals, curves, and phase fields still
  hold. Using the uniform bounds on $\gamma_\delta$, $\gamma_\delta'$,
  and $u_\delta$ we obtain convergence of $\gamma_\delta$ in
  $W^{1,p}(I; \bbR^2)$ and $u_\delta$ in $L^p(I)$. Since number and
  height of interfaces are not affected, the interface energy remains
  unchanged.

  The phase integral constraint is easily recovered by moving an
  existing interface slightly or introducing one or finitely many new
  ones at a height that vanishes with $\delta \to 0$.
\end{proof}

\begin{remark}
  The construction in the proof of Lemma
  \ref{lem:approx-simple-curves} can also be done at $\partial I$.
  Hence, the Lemma comprises the result that any $\gamma \in \calC$
  can be approximated by curves from $\calC \cap W^{2,2}(I;\bbR^2)$.
\end{remark}


The next step is to find a recovery sequence for membranes
$(\gamma,u)$ as constructed in Lemma \ref{lem:approx-simple-curves}.
To this end, let $s \in \set{y=0}$ and fix $J \Subset I$ such that
$\overline{J} \cap (\set{y=0} \cup S_u) = \set{s}$ and $\gamma$ is a
vertical line in $J \cap \set{t>s}$ and $J \cap \set{t<s}$. For
simplicity of notation we assume again $s=0$, $\gamma(0) = (0,0)$ and
$|\gamma'|=1$.

A $\delta$-catenoid is the surface generated by a $\delta$-catenary
whose unit speed parametrisation $c_\delta = (i_\delta,j_\delta) \in
C^\infty(\bbR;\bbR^2)$ is given by
\begin{equation*}
  i_{\delta}(t) = \delta \arcsinh \tfrac{t}{\delta},
  \qquad
  j_{\delta}(t) = \sqrt{\delta^2+t^2}.
\end{equation*}
Its principal curvatures are
\begin{equation*}
  - \kappa_{1,\delta}(t)
  =
  \kappa_{2,\delta}(t)
  =
  \frac{\delta}{\delta^2+t^2},
\end{equation*}
and thus we have
\begin{equation}
  \label{eq:catenoid-sec-fform-bound}
  \int_{M_{c_\delta}(a,b)} |B_\delta|^2 \, d\mu_\delta
  =
  4\pi \int_{a}^{b} \frac{\delta^2}{(\delta^2+t^2)^2} \sqrt{\delta^2+t^2} \,d t
  =
  4\pi \left. \frac{t}{\sqrt{\delta^2+t^2}} \right|_a^b
  \leq
  8\pi
\end{equation}
for any $\delta>0$ and all $a,b \in \bbR$, $a<b$. Since the
$\delta$-catenary satisfies $c_\delta(0) = (0,\delta)$ and
$c_\delta'(0) = (1,0)$, it suffices to study a construction for $J
\cap \set{t\geq0}$ and join it with its mirrored counterpart in $J
\cap \set{t\leq0}$.

\begin{lemma}
  \label{lem:catenoid-and-circle}
  Assume that $\gamma$ is a vertical line segment in $J=(0,t_0)$, that
  is $\gamma(t)=(0,t)$ in $J$.  Then for all sufficiently small
  $\delta$ depending only on $\gamma$ there is a curve $\gamma_\delta
  =(x_\delta,y_\delta) \in W^{2,2}(J;\bbR^2)$ such that
  \begin{itemize}
  \item $\gamma_\delta$ satisfies $y_\delta>0$, $x_\delta' \geq 0$ in
    $J$ and there is $J_\delta=(0,t_\delta)$ for some $t_\delta \to
    0$ as $\delta \to 0$ such that $\gamma_\delta$ is a vertical line
    segment in $J \sm J_\delta$ and
    \begin{equation*}
      |\gamma_\delta'| =
      \begin{cases}
        1            &\text{if } t = J \sm \widetilde J,\\
        1 + r_\delta &\text{if } t \in \widetilde J,
      \end{cases}
    \end{equation*}
    where $r_\delta \to 0$ in $W^{1,2}(\widetilde J) \cap
    C^0_c(\widetilde J)$ as $\delta \to 0$ and $\widetilde J \subset J
    \sm J_\delta$;
  \item at the end points of $\gamma_\delta(J)$ we have
    \begin{equation*}
      \gamma_\delta(0) = (0,\delta),
      \quad
      \gamma_\delta'(0) = (1,0),
      \quad
      \gamma_\delta(t_0) = (x(t_0) + o(1), y(t_0)),
      \quad
      \gamma_\delta'(t_0) = \gamma'(t_0);
    \end{equation*}
  \item $\gamma_\delta \to \gamma$ in $W^{1,p}(J;\bbR^2)$ for any $p
    \in [1,\infty)$ as $\delta \to 0$;
  \item $\calA_{\delta}(J) = \calA(J) + o(1)$ and $\displaystyle
    \int_{M_{\delta}(J)} u \,d\mu_\delta = \int_{M_\gamma(J)} u \,d\mu
    + o(1)$ as $\delta \to 0$;
  \item $\displaystyle \sup_{\delta>0} \int_{M_{\delta}(J)} |B_\delta|^2
    \,d\mu_\delta < \infty$; and
  \item $\displaystyle \int_{M_{\delta}(J)}
    k(u) (H_\delta-H_{s}(u))^2 \,d\mu_\delta
    \to
    \int_{M_\gamma(J)} k(u) (H-H_{s}(u))^2 \,d\mu$ as $\delta \to 0$.
  \end{itemize}
\end{lemma}

\begin{proof}
  With $s_\delta$ and $t_\delta$ to be determined, we replace $\gamma$
  by a $\delta$-catenary in some interval $[0,s_\delta)$ and a segment
  of a circle of radius $1$ in $[s_\delta,t_\delta)$, which connects
  the catenary and the shifted original curve; see
  Figure~\ref{fig:curve-constructions}(3). Writing
  \begin{equation*}
    k_\delta(t) = \sin(t-b_\delta) + \hat k_\delta
    \qquad\text{and}\qquad
    l_\delta(t) = -\cos(t-b_\delta) + \hat l_\delta
  \end{equation*}
  for the coordinates of the circle and fixing $\delta$ and $s_\delta$
  for the moment, we aim to determine $t_\delta$, $b_\delta$, $\hat
  k_\delta$, $\hat l_\delta$, and shifts $m_\delta$, $n_\delta$ such
  that
  \begin{equation*}
    \gamma_\delta(t) =
    \begin{cases}
      c_\delta(t)
      &\text{if } 0 \leq t < s_\delta,
      \\
      (k_\delta(t),l_\delta(t))
      &\text{if } s_\delta \leq t < t_\delta,
      \\
      (m_\delta,t + n_\delta)
      &\text{if } t_\delta \leq t
    \end{cases}
  \end{equation*}
  is continuously differentiable at $s_\delta$ and $t_\delta$. The
  corresponding conditions are
  \begin{align*}
    k_\delta'(s_\delta) &= i_\delta'(s_\delta),
    &
    l_\delta'(s_\delta) &= j_\delta'(s_\delta),
    &
    k_\delta(s_\delta) &= i_\delta(s_\delta),
    &
    l_\delta(s_\delta) &= j_\delta(s_\delta),
    \\
    k_\delta'(t_\delta) &= 0,
    &
    l_\delta'(t_\delta) &= 1,
    &
    k_\delta(t_\delta) &= m_\delta,
    &
    l_\delta(t_\delta) &= t + n_\delta,
  \end{align*}
  and a short calculation shows
  \begin{align*}
    b_\delta
    &= s_\delta - \arctan (s_\delta/\delta),
    &
    t_\delta
    &= \pi/2 + b_\delta,
    \\
    \hat k_\delta
    &= i_\delta(s_\delta) - \sin(s_\delta-b_\delta),
    &
    \hat l_\delta
    &= j_\delta(s_\delta) + \cos(s_\delta-b_\delta),
    \\
    m_\delta
    &= 1 + \hat k_\delta,
    &
    n_\delta
    &= \hat l_\delta - t_\delta.
  \end{align*}
  If we let $s_\delta = \delta^\beta$ for some $\beta \in (0,3/4)$, we
  find
  \begin{equation*}
    t_\delta \sim \delta^\beta + \delta^{1-\beta},
    \qquad
    m_\delta \sim \delta \ln\delta^{\beta-1},
    \qquad
    n_\delta \sim \delta^\beta + \delta^{1-\beta},
  \end{equation*}
  that is, the catenary and circle vanish in the limit $\delta \to 0$,
  and therefore $\gamma_\delta \to \gamma$ in $W^{1,p}(J; \bbR^2)$ for
  $p \in [1,\infty)$ and $\calA_{\gamma_\delta}(J) \to
  \calA_\gamma(J)$.  A more precise estimate shows
  $\calA_{\gamma_\delta}(J) = \calA_\gamma(J) +
  O(\delta+\delta^{2\beta} + \delta^{2-2\beta})$ and the same order
  for the error in the phase integral constraint.

  The principal curvatures of the circle segment are
  \begin{equation*}
    \kappa_1 = -1
    \qquad \text{and} \qquad
    \kappa_2 = \frac{\cos(t-b_\delta)}{\hat l_\delta - \cos(t-b_\delta)},
  \end{equation*}
  thus the second fundamental form is estimated by
  \begin{align*}
    \frac{1}{2\pi}\int_{M_{\gamma_\delta}(s_\delta,t_\delta)}
    |B_\delta|^2 \,d\mu_\delta
    &\leq
    \left( 1 + \frac{\cos^2(s_\delta-b_\delta)}{j_\delta(s_\delta)}
    \right)
    \cdot (t_\delta - s_\delta)
    \\
    &=
    \left( 1 + \frac{\delta^2}{(\delta^2+\delta^{2\beta})^{3/2}}
    \right)
    \cdot
    \left( \pi/2 - \arctan \delta^{\beta-1} \right)
    \\
    &\sim
    \left( 1 + \delta^{2-3\beta} \right) \cdot \delta^{1-\beta}
    =
    \delta^{1-\beta} + \delta^{3-4\beta},
  \end{align*}
  which tends to $0$ as $\delta \to 0$ due to $0 < \beta <
  3/4$. By~\eqref{eq:catenoid-sec-fform-bound} the second fundamental
  forms of $M_{\gamma_\delta}$ are thus uniformly bounded in
  $L^2$. Similarly, the integral of $H_{\gamma_\delta}^2$ over the
  circle segments vanishes in the limit $\delta \to 0$, and since
  $H_{\gamma_\delta}(t) = 0$ for $t<s_\delta$ and
  $H_{\gamma_\delta}(t) = H_\gamma(t)$ for $t>t_\delta$, we obtain
  \begin{equation*}
    \int_{M_{\gamma_\delta}(J)}
    k(u) (H_\delta-H_{s}(u))^2 \,d\mu_{\gamma_\delta}
    \to
    \int_{M_{\gamma}(J)} k(u) (H-H_{s}(u))^2 \,d\mu
  \end{equation*}
  as $\delta \to 0$.

  To dispose of the shift $n_\delta$ in $y$-direction, fix $\widetilde
  J \Subset J$ with $t_\delta < \inf \widetilde J$ for all
  sufficiently small $\delta$ and a function $f \in
  C_c^\infty(\widetilde J)$ with $\int_J f \,d t = 1$. The perturbed
  curve $\widetilde \gamma_\delta = (\widetilde x_\delta, \widetilde
  y_\delta) = ( x_\delta, y_\delta - n_\delta F )$, where $F(t) =
  \int_0^t f(s) \,d s$, has the desired end point $y$-coordinate
  \begin{equation*}
    \widetilde y_\delta(t_0) = y_\delta(t_0) - n_\delta = y(t_0).
  \end{equation*}
  Since $|n_\delta| \to 0$ as $\delta \to 0$, the perturbation
  vanishes in any function space to which $\gamma_\delta$ belongs.
  The second fundamental form is still uniformly bounded in $L^2(J)$,
  because for all small $\delta$ the perturbation is supported in a
  vertical line segment of $\gamma_\delta$, where both principal
  curvatures are equal to $0$. Moreover, we have $\inf_{\widetilde J}
  \widetilde y_\delta > 0$ and the error in the area constraint in $J$
  and in the phase integral constraint are of the same order as above.
\end{proof}


Now it is easy to see why $k_G(u)$ has to be adapted near points on
the axis of revolution, even if $u$ does not have a jump
there. Consider for instance a surface $M_\gamma$ consisting of two
balls connected at the axis of revolution and assume that $k_G(u)
\equiv -1$. Then $\int_{M_\gamma} k_G K \,d\mu = -8\pi$, while for any
approximation $M_\delta$ with one component we have $\int_{M_\delta}
k_G K_\delta \,d\mu_\delta = - 4\pi$. Since the mean curvature
integral converges, the total curvature energy drops in the limit
$\eps \to 0$. Changing the phase field such that $k_G(u_\eps) = 0$
compensates for this effect.

In the following corollary we combine all previous constructions and
apply the additional phase field change. This finishes the proof of
the upper bound.

\begin{corollary}
  \label{cor:curve-approx}
  Let $(\gamma,u) \in \calC \times \calP$ have finitely many
  interfaces and vertical line segments near component boundaries in
  $I$. Then there are $(\gamma_\eps, u_\eps) \in
  \calC_\eps \times \calP_\eps$ such that $\gamma_\eps \to \gamma$ in
  $W^{1,1}(I;\bbR^2)$, $u_\eps \to u$ in $L^1(I)$ and
  $\calE_\eps(\gamma_\eps,u_\eps) \to \calE(\gamma,u)$ as $\eps \to
  0$.
\end{corollary}

\begin{proof}
  Let $\set{y=0} \cap I = \set{s_1,\ldots,s_n}$ where
  $n=N_\gamma-1$. We employ Lemma \ref{lem:catenoid-and-circle} and
  its mirrored version successively for each $s_k$, taking the global
  shifts in $x$-direction into account.  Since the parameter $\delta>0$
  in Lemma \ref{lem:catenoid-and-circle} is independent of $\eps$, we
  can choose it so small that the curve replacement, apart from
  $x$-shifts and the small perturbation in Lemma
  \ref{lem:catenoid-and-circle}, takes place in intervals $J_{k,\eps}$
  around $s_k$ of length at most $\eps$ and with area and phase
  constraint error bounded by $\sqrt{\eps}$. The result for
  sufficiently small $\eps$ is a sequence $(\gamma_\eps)$ that
  converges to $\gamma$ in $W^{1,p}(I; \bbR^2)$ for any $p \in
  [1,\infty)$ and satisfies
  \begin{equation}
    \label{eq:mc-int-conv}
    \int_{M_{\eps}} k(u) (H_\eps-H_{s}(u))^2 \,d\mu_\eps
    \to
    \int_{M_{\gamma}} k(u) (H-H_{s}(u))^2 \,d\mu
    \qquad\text{as }\eps\to0.
  \end{equation}

  We construct $u_\eps$ by first replacing $u$ around $S_u$ with the
  local recovery sequences from Section
  \ref{sec:local-interf-recov}. Additionally, we set $u_\eps=0$ in
  each $J_{k,\eps}$ and at $\partial J_{k,\eps}$ we make a transition
  to $\pm 1$ exactly as in Section \ref{sec:local-interf-recov}. The
  phase field energy of $J_{k,\eps}$ is
  \begin{equation*}
    2\pi \int_{J_{k,\eps}} \frac{1}{\eps} W(0) y_\eps \,d t
    \leq
    2\pi W(0) \sup_{J_{k,\eps}} y_\eps,
  \end{equation*}
  and the costs of the transitions near $\partial J_{k,\eps}$ are
  bounded by $2\pi \sigma \sup_{\partial J_{k,\eps}} y_\eps$ as shown
  in Lemma \ref{lem:phase-field-limsup}. By construction of
  $\gamma_\eps$ both vanish as $\eps\to0$. Moreover
  \begin{equation*}
    \int_{M_\eps} k_G(u_\eps) K_\eps \,d\mu_\eps
    =
    \int_{M_\eps(I \sm \cup J_{k,\eps})} k_G(u_\eps) K_\eps \,d\mu_\eps
    \to
    \int_{M_\gamma} k_G(u) K \,d\mu,
  \end{equation*}
  since $u_\eps \to u$ in $L^1(I)$, $K_\eps = K$ outside $\cup
  J_{k,\eps}$, and $\int_{M} |K| \,d\mu$ is finite. Also,
  \eqref{eq:mc-int-conv} still holds with $u$ replaced by $u_\eps$ on
  the left hand side. Hence, we find $\calE_\eps(\gamma_\eps,u_\eps)
  \to \calE(\gamma,u)$ as $\eps \to 0$.

  Finally, to obtain a membrane in $\calC_\eps \times \calP_\eps$ we
  once more have to correct the constraints. For the area this is done
  as in Lemma \ref{lem:approx-simple-curves}, while the error in the
  phase integral, which is introduced by the constructions at the axis
  of revolution, can be corrected as in Lemma
  \ref{lem:phase-field-limsup}, as it is of order $\sqrt{\eps}$. The
  result is a membrane $(\gamma_\eps,u_\eps)$ that satisfies all
  conditions of $\calC_\eps \times \calP_\eps$ except for the constant
  speed requirement. However, since by construction $|\gamma_\eps'| =
  1 + o(1)$ and the perturbation vanishes in $W^{1,2}$, the constant
  speed reparametrisations converge to the identity in $W^{2,2}(I)$
  and the properties of $(\gamma_\eps, u_\eps)$ carry over to
  reparametrised curve and phase field.
\end{proof}


\section{Some generalisations}
\label{sec:generalisations}

We conclude the paper with some extensions of
Theorem~\ref{thm:gamma-conv}. First of all, the proof is easily
adapted to non-symmetric potentials $W$. In this case, one considers
the complete optimal profile $p$ in Section
\ref{sec:local-interf-recov} and uses the appropriate side in the
connections to regions $\set{u_\eps=0}$ in Corollary
\ref{cor:curve-approx}. One may also consider potentials like $W(u) =
(1-u)^2$ and drop the phase integral constraint for $u_\eps$. Then
there is only one lipid phase, and $u_\eps$ is merely an auxiliary
variable that allows the recovery of topological changes at the axis
of revolution in the limit.


The constraint of prescribed area for the approximate setting can be
relaxed to
\begin{equation*}
  0
  <
  \inf_{\gamma \in \calC_\eps} \calA_\gamma
  \leq
  \sup_{\gamma \in \calC_\eps} \calA_\gamma
  <
  \infty,
\end{equation*}
and the arguments for equi-coercivity and lower bound still apply. It
can be incorporated as penalty term in the energy, for instance by
$\left( \calA_{\gamma} - A_0 \right)^2 / \eps$ or any other scale of
$\eps$, because we have recovered it exactly. In the same way, the
phase integral constraint can be replaced by a penalty term. Other
constraints that change continuously under the convergence proved in
Lemmas \ref{lem:coerc-curves} and \ref{lem:coerc-phase-fields} can
also be imposed, for instance on the enclosed volume $\calV_\gamma(M)
= \pi \int_{M_\gamma} x' y^2 \,d t$. Of course, constraints have to be
compatible, so that the set of admissible membranes is non-empty.


The arguments in Section \ref{sec:proof} also apply to open surfaces
of revolution generated by curves $\gamma=(x,y) \colon I \to \bbR
\times \bbR_{>0}$ with prescribed boundary conditions for $\gamma$ and
$\gamma'$ at $\partial I$. The curve length is then controlled by
energy, area, and boundary conditions due to
\begin{equation*}
  2\pi \calL_\gamma
  \leq
  \int_{M_\gamma} |H| \,d\mu + 2\pi \widetilde\phi y|_{\partial I},
\end{equation*}
where $\widetilde\phi$ is the tangent angle as in the proof of
Lemma~\ref{lem:surf:length-bound}. Note that the boundary condition
for $\gamma'$ at $\partial I$ is preserved as $\eps \to 0$ because
$y>0$ at $\partial I$. Furthermore, since $\|\widetilde\phi\|_\infty
\leq \pi/2$, it is also possible to weaken the boundary conditions to
requiring a uniform $L^\infty$-bound on $y$ at $\partial I$. Such a
bound can for instance be derived from uniformly bounded energy
$\calE_\eps + \calG$, where
\begin{equation*}
  \calG(\gamma)
  =
  \overline{\sigma} \int_{M_\gamma(\partial I)} d\calH^1
  =
  2\pi \overline{\sigma} \sum_{s \in \partial I} y(s)
\end{equation*}
with a constant line tension $\overline{\sigma}$. Since $\calG$ is
continuous with respect to curve convergence in $C^0$, its presence
does not influence the $\Gamma$-convergence. The limit energy
$\calE+\calG$ models open lipid membranes; see for instance
\cite{SaTaTaHo95,TuOu03,DuWa08} for experimental observations,
modelling and numerical simulations of single-phase open membranes,
respectively.


\section*{Acknowledgements}

The author thanks Barbara Niethammer and Michael Herrmann for the
countless number of helpful remarks and discussions about the results
presented here. This work was supported by the EPSRC Science and
Innovation award to the Oxford Centre for Nonlinear PDE
(EP/E035027/1).


\bibliography{paper}
\bibliographystyle{abbrv}

\end{document}